\documentclass{amsart}

\usepackage[english]{babel}
\usepackage[latin1]{inputenc}
 \usepackage[T1]{fontenc}
 \usepackage{lmodern}
 
\usepackage{csquotes}
\usepackage{tikz}
\usepackage[all]{xy}
\usepackage{amsmath,amssymb,amsthm}
\usepackage{marvosym}
\usepackage{fullpage}
\usepackage{xy}
\usepackage{enumerate}
\usepackage{mathrsfs}

\newcommand{\Q}{\mathbb{Q}}
\newcommand{\C}{\mathbb{C}}
\newcommand{\PP}{\mathbb{P}}

\newcommand{\isomto}{\overset{\simeq}{\rightarrow}}
\newcommand{\K}{\mathcal{K}}

\newcommand{\wE}{{}_{\mathrm{w}}\!E}
\newcommand{\gr}{\mathrm{gr}}

\newcommand{\s}{\mathscr{S}}
\newcommand{\td}[1]{\widetilde{#1}}

\renewcommand{\leq}{\leqslant}
\renewcommand{\geq}{\geqslant}

\newtheorem{thm}{Theorem}[section]

\newtheorem{lem}[thm]{Lemma}
\newtheorem{prop}[thm]{Proposition}

\theoremstyle{remark}
\newtheorem{rem}[thm]{Remark}

\theoremstyle{remark}
\newtheorem{ex}[thm]{Example}

\theoremstyle{definition}
\newtheorem{defi}[thm]{Definition}

\title[The Orlik-Solomon model]{The Orlik-Solomon model for hypersurface arrangements}

\author{Cl\'{e}ment Dupont}

\address{Max-Planck Institut f\"{u}r Mathematik\\
Vivatsgasse 7\\
53111 Bonn, Germany.}

\email{cdupont@mpim-bonn.mpg.de}

\begin{document}

\begin{abstract}
  We develop a model for the cohomology of the complement of a hypersurface arrangement inside a smooth projective complex variety. This generalizes the case of normal crossing divisors, discovered by P. Deligne in the context of the mixed Hodge theory of smooth complex varieties. Our model is a global version of the Orlik-Solomon algebra, which computes the cohomology of the complement of a union of hyperplanes in an affine space. The main tool is the complex of logarithmic forms along a hypersurface arrangement, and its weight filtration. Connections with wonderful compactifications and the configuration spaces of points on curves are also studied.
\end{abstract}

\maketitle

\section{Introduction}

		Let~$X$ be a complex manifold of dimension~$n$. A hypersurface arrangement in $X$ is a union 
		$$L=L_1\cup\cdots\cup L_l$$
		of smooth hypersurfaces~$L_i\subset X$,~$i=1,\ldots,l$, that locally looks like a union of hyperplanes in~$\C^n$: around each point of~$X$ we can find a system of local coordinates in which each~$L_i$ is defined by a linear equation.\\
		This generalizes the notion of a (simple) normal crossing divisor: a hypersurface arrangement is a normal crossing divisor if the local linear equations defining the~$L_i$'s are everywhere linearly independent; in other words, if we can always choose local coordinates~$(z_1,\ldots,z_n)$ such that~$L$ is locally defined by the equation~$z_1\cdots z_r=0$ for some~$r$.\\
		Besides normal crossing divisors, examples of hypersurface arrangements include unions of hyperplanes in a projective space~$\PP^n(\C)$, or unions of diagonals~$\Delta_{i,j}=\{y_i=y_j\}\subset Y^n$ inside the~$n$-fold cartesian product of a Riemann surface~$Y$. The class of hypersurface arrangements is also closed under certain blow-ups.\\
		
		The aim of this article is to define and study a model~$M^\bullet(X,L)$ for the cohomology algebra over~$\Q$ of the complement~$X\setminus L$ of a hypersurface arrangement, when~$X$ is a smooth projective variety over~$\C$.\\
		Our model, which we call the Orlik-Solomon model, has \textit{combinatorial} inputs coming from the theory of hyperplane arrangements (the local setting) and \textit{geometric} inputs coming from the cohomology of smooth hypersurface complements in a smooth projective variety (the global setting). Roughly speaking, it is the \textit{direct product} of two classical tools related to these two situations, that we first recall.
		\begin{itemize}
		\item \textit{Combinatorics: the Orlik-Solomon algebra.}
		Let~$L$ be a union of hyperplanes in~$\C^n$ that contain the origin, and call any multiple intersection of hyperplanes of~$L$ a stratum of~$L$. The strata of~$L$ form a poset which is graded by the codimension of the strata, and denoted by~$\s_\bullet(L)$. In~\cite{orliksolomon}, Orlik and Solomon introduced~$\Q$-vector spaces~$A_S(L)$ for every stratum~$S$, and gave the direct sum
		\begin{equation}\label{introdirectsumS}
		A_\bullet(L)=\bigoplus_{S\in \s_\bullet(L)}A_S(L)
		\end{equation}
		the structure of a graded algebra, via product maps
		\begin{equation}\label{introprodS}
		A_S(L)\otimes A_{S'}(L)\rightarrow A_{S\cap S'}(L).
		\end{equation}		
		Furthermore, there are natural morphisms
		\begin{equation}\label{introdiffS}
		A_S(L)\rightarrow A_{S'}(L)
		\end{equation}		
		for any inclusion~$S\subset S'$ of strata of~$L$ such that~$\mathrm{codim}(S')=\mathrm{codim}(S)-1$. The crucial fact is that the Orlik-Solomon algebra is a combinatorial object, which means that it only depends on the poset of the strata of~$L$. We now recall the classical Brieskorn-Orlik-Solomon theorem (see Theorem \ref{bos} for a more precise statement). Here~$H^\bullet(\C^n\setminus L)$ denotes the cohomology of the complement~$\C^n\setminus L$ with rational coefficients.
		
		\begin{thm}[Brieskorn-Orlik-Solomon]
		We have an isomorphism of graded algebras~$$H^\bullet(\C^n\setminus L)\cong A_\bullet(L).$$
		\end{thm}
		
		One may define an Orlik-Solomon algebra~$A_\bullet(L)$ for~$L$ any hypersurface arrangement inside a complex manifold~$X$. We still have a direct sum decomposition (\ref{introdirectsumS}), with~$\s_\bullet(L)$ the graded poset of strata of~$L$, as well as product maps (\ref{introprodS}) and natural morphisms (\ref{introdiffS}). As in the local case, the Orlik-Solomon algebra~$A_\bullet(L)$ only depends on the poset of strata of~$L$. It is functorial with respect to~$(X,L)$ in the sense that any holomorphic map~$\varphi:X\rightarrow X'$ such that~$\varphi^{-1}(L')\subset L$ induces a map of graded algebras~$A_\bullet(\varphi):A_\bullet(L')\rightarrow A_\bullet(L)$.\\
		
		\item \textit{Geometry: the Gysin long exact sequence.}
		For a smooth hypersurface~$V$ inside a smooth projective variety~$X$ over~$\C$, the Gysin morphisms of the inclusion~$V\subset X$ are the morphisms~$H^{k-2}(V)(-1)\rightarrow H^k(X)$ , where~$(-1)$ denotes a Tate twist, obtained as the Poincar\'{e} duals of the natural morphisms
		$H^{2n-k}(X)\rightarrow H^{2n-k}(V)$
		where~$n=\mathrm{dim}_\C(X)$. They fit into a long exact sequence, called the Gysin long exact sequence:
		\begin{equation}\label{introgysinles}
		\cdots \rightarrow H^{k-2}(V)(-1)\rightarrow H^k(X) \rightarrow H^k(X\setminus V)\rightarrow H^{k-1}(V)(-1) \rightarrow \cdots
		\end{equation}
		It is worth noting that the connecting homomorphisms~$H^k(X\setminus V)\rightarrow H^{k-1}(V)(-1)~$ are residue morphisms, which are easily described using logarithmic forms.
		\end{itemize}		
		
		We can now state our main theorem (see Theorem \ref{maintheorem} for more precise statements).
		
		\begin{thm}\label{intromaintheorem}
		Let~$X$ be a smooth projective variety over~$\C$ and~$L$ be a hypersurface arrangement in~$X$.
		\begin{enumerate}
		\item For integers~$q$ and~$n$ let us consider
		$$M_q^n(X,L)=\bigoplus_{S\in\s_{q-n}(L)}H^{2n-q}(S)(n-q)\otimes A_S(L)$$
		where~$(n-q)$ is a Tate twist, viewed as a pure Hodge structure of weight~$q$. Then the direct sum 
		$$M^\bullet(X,L)=\bigoplus_{q}M^\bullet_q(X,L)$$
		has the structure of a differential graded algebra (dga) in the (semi-simple) category of split mixed Hodge structures over~$\Q$. The product in~$M^\bullet(X,L)$ is induced by the product maps (\ref{introprodS}) of the Orlik-Solomon algebra and the cup-product on the cohomology of the strata. The differential in~$M^\bullet(X,L)$ is induced by the natural morphisms (\ref{introdiffS}) and the Gysin morphisms
		$$H^{2n-q}(S)(n-q)\rightarrow H^{2n-q+2}(S')(n+1-q)$$ of the inclusions of  strata~$S\subset S'$. The dga~$M^\bullet(X,L)$ is functorial with respect to~$(X,L)$ in the sense explained above.
		\item The dga~$M^\bullet(X,L)$ is a model for the cohomology of~$X\setminus L$ in the following sense: we have isomorphisms of pure Hodge structures over~$\Q$
		$$\gr^W_q H^n(X\setminus L)\cong H^n(M_q^\bullet(X,L))$$ which are compatible with the algebra structures, and functorial with respect to~$(X,L)$.
		\end{enumerate}
		\end{thm}
		
		The precise definition of the Orlik-Solomon model~$M^\bullet(X,L)$ is given in \S\ref{defM}. Theorem \ref{intromaintheorem} generalizes the case of normal crossing divisors, which is due to P. Deligne~\cite{delignehodge2}, see also~\cite[8.35]{voisin}, as a by-product of the definition of the mixed Hodge structure on the cohomology of smooth varieties over~$\C$. The Orlik-Solomon model appears as the first page of a spectral sequence, called the Orlik-Solomon spectral sequence.\\
		
		Before we describe the proof of Theorem \ref{intromaintheorem} and some of its applications, we mention that it completes a result by E. Looijenga~\cite[\S 2]{looijenga} who first considered the Orlik-Solomon spectral sequence. Our approach is totally different, with a prominent use of differential forms. In particular, we introduce a complex of logarithmic differential forms (see \S\ref{introseclog} in this Introduction) that should have applications in other situations. Concretely, the main advantages of the use of differential forms are the following.
		\begin{enumerate}
		\item It allows us to prove the functoriality of the Orlik-Solomon model, whereas Looijenga's spectral sequence cannot be easily proved to be functorial. This is crucial when discussing the behaviour of the Orlik-Solomon model with respect to blow-ups (see \S\ref{introsecwonderful} in this Introduction and \S\ref{sectionwonderful}). As a consequence, we are able to reconcile Kriz's and Totaro's approaches on models for configuration spaces of points on curves (see \S\ref{introsecconf} in this Introduction and \S\ref{seccompkriz}).
		\item It makes the multiplicative structure of the Orlik-Solomon model transparent and closer in spirit to the classical Brieskorn-Orlik-Solomon theorem.
		\item Our approach is more down-to-earth in that we prove that the Orlik-Solomon spectral sequence is compatible with Hodge structures using only mixed Hodge theory \textit{\`{a} la} Deligne. With Looijenga's formalism, one would have to use Saito's theory of mixed Hodge modules (in this direction, see also~\cite{getzlerresolvingmhm}): indeed, his spectral sequence is defined out of a complex of sheaves built out of the constructible sheaves~$i_!i^!\mathbb{Q}$ for~$i$ a closed immersion, hence it is not immediate that it is compatible with mixed Hodge theory.
		\end{enumerate}
		
		\subsection{Wonderful compactifications}\label{introsecwonderful}
		
		We should say a word on the usefulness of the generalization from normal crossing divisors to hypersurface arrangements. Indeed, Deligne's approach relies on the fact that any smooth variety over~$\C$ can be viewed as the complement of a normal crossing divisor inside a smooth projective variety, using Nagata's compactification theorem and Hironaka's resolution of singularities. Thus the case of normal crossing divisors is (in principle) sufficient to give a model for the cohomology of \textit{any} smooth variety over~$\C$.\\
		In the framework of Theorem \ref{intromaintheorem}, we may even produce, following~\cite{fultonmcpherson,deconciniprocesi,hu,li}, an explicit sequence of blow-ups (see Theorem \ref{seqblowups})
		$$\pi:\td{X}\rightarrow X$$
		sometimes called a ``wonderful compactification'', that transforms~$L$ into a normal crossing divisor~$\td{L}=\pi^{-1}(L)$ inside~$\td{X}$ and induces an isomorphism 
		$$\pi:\td{X}\setminus\td{L}\isomto X\setminus L.$$ 
		Thus Deligne's special case of Theorem \ref{intromaintheorem} applied to~$(\td{X},\td{L})$ gives a model~$M^\bullet(\td{X},\td{L})$ for the cohomology of~$X\setminus L$. The functoriality of our construction gives a quasi-isomorphism of differential graded algebras
		\begin{equation}\label{qisMpi}
		M^\bullet(\pi):M^\bullet(X,L)\stackrel{\sim}{\rightarrow} M^\bullet(\td{X},\td{L})
		\end{equation}
		that we may compute explicitly (see Theorem \ref{formulaMpi}).\\
		The model~$M^\bullet(X,L)$ has three advantages over~$M^\bullet(\td{X},\td{L})$. Firstly, it is in general smaller ($M^\bullet(\pi)$ is always injective). Secondly, its definition only uses geometric and combinatorial information from the pair~$(X,L)$ without having to look at the blown-up situation~$(\td{X},\td{L})$. Thirdly, it is functorial with respect to~$(X,L)$.\\
		
		Along with the work of Morgan~\cite[Theorem 10.1]{morganalgebraictopology}, the quasi-isomorphism (\ref{qisMpi}) implies that~$M^\bullet(X,L)$ is a model of the space~$X\setminus L$ in the sense of rational homotopy theory (Theorem \ref{thmrationalhomotopy}).
		
		\subsection{Configuration spaces of points on curves}\label{introsecconf}
	
		Let~$Y$ be a compact Riemann surface and~$n$ be an integer. For all~$1\leq i<j\leq n$ we have a diagonal
		$$\Delta_{i,j}=\{y_i=y_j\}\subset Y^n$$ inside the~$n$-fold cartesian product of~$Y$. Any union of~$\Delta_{i,j}$'s then defines a hypersurface arrangement in~$Y^n$. For example, if we consider the union of all diagonals, the complement is the configuration space of~$n$ ordered points in~$Y$:
		$$C(Y,n)=\{(y_1,\ldots,y_n)\in Y^n\,\,|\,\,y_i\neq y_j\ \,\, \textnormal{for}\,\, i\neq j\}.$$
		Theorem \ref{intromaintheorem} hence gives an Orlik-Solomon model for the cohomology of~$C(Y,n)$. This model is isomorphic to the one independently found by I. Kriz~\cite{kriz} and B. Totaro~\cite{totaroconf}, as we prove in Theorem \ref{compbloch}.\\
		
		On the one hand, our method is close to Totaro's, since the Orlik-Solomon spectral sequence that we are considering in \S\ref{gysinss} is the Leray spectral sequence of the inclusion~$j:X\setminus L\hookrightarrow X$. On the other hand, the functoriality of our constructions implies that there exists a quasi-isomorphism~$M^\bullet(\pi)$ associated to any wonderful compactification~$\pi$; in \S\ref{seccompkriz} we prove that this quasi-isomorphism is exactly the one used by Kriz to prove the main result of~\cite{kriz}. Hence, our method reconciles Kriz's and Totaro's approaches in the case of curves.\\
		
		As a natural generalization, we consider the union of only certain diagonals~$\Delta_{i,j}$. Such a generalization has been recently studied by S. Bloch~\cite{blochtreeterated}, who gives a model in the spirit of Kriz and Totaro's model. We prove that this model is also isomorphic to our Orlik-Solomon model.	\\
		
		Shortly after a preprint of the present article was released, C. Bibby independently showed \cite{bibbyabelian} the existence of the Orlik-Solomon model following Totaro's approach, and applied it to the case of abelian arrangements. In \cite{bibbyhilburnchordal}, C. Bibby and J. Hilburn used the Orlik-Solomon model to study the homotopy-theoretic properties of certain configuration spaces of points on curves.
		
	\subsection{Logarithmic forms and mixed Hodge theory}\label{introseclog}
	
		We now discuss the proof of Theorem \ref{intromaintheorem}. Our approach follows Deligne's proof of the case of normal crossing divisors, hence makes extensive use of logarithmic forms and the formalism of mixed Hodge structures.\\
		Let~$X$ be a smooth projective variety and~$L=L_1\cup\cdots\cup L_l$ be a hypersurface arrangement in~$X$. The first task is to define a complex of sheaves on~$X$, denoted by~$\Omega^\bullet_{\langle X,L\rangle}$, of meromorphic forms on~$X$ with logarithmic poles along~$L$. In local coordinates where each~$L_i$ is defined by a linear equation~$f_i=0$, a section of~$\Omega^\bullet_{\langle X,L\rangle}$ is a meromorphic differential form on~$X$ which is a linear combination over~$\C$ of forms of the type
		\begin{equation}\label{eqlogformsncd1}
		\eta\wedge\dfrac{df_{i_1}}{f_{i_1}}\wedge\cdots\wedge\dfrac{df_{i_s}}{f_{i_s}}
		\end{equation}
		with~$\eta$ a holomorphic form and~$1\leq i_1< \cdots< i_s\leq l$. It has to be noted that the complex~$\Omega^\bullet_{\langle X,L\rangle}$ is in general a strict subcomplex of the complex~$\Omega^\bullet_X(\log L)$ introduced by Saito~\cite{saitologarithmic}, even though the two complexes coincide in the case of a normal crossing divisor.\\
		
		The main point of the complex~$\Omega^\bullet_{\langle X,L\rangle}$ is that it computes the cohomology of the complement~$X\setminus L$. More precisely, if we denote by~$j:X\setminus L\hookrightarrow X$ the open immersion of the complement of~$L$ inside~$X$, we prove the following theorem (Theorem \ref{qisglobal}).
		
		\begin{thm}\label{introqis}
		The inclusion~$\Omega^\bullet_{\langle X,L\rangle}\hookrightarrow j_*\Omega^\bullet_{X\setminus L}$ is a quasi-isomorphism, and hence induces isomorphisms
		\begin{equation}\label{isoiso}\mathbb{H}^n(\Omega^\bullet_{\langle X,L\rangle})\cong H^n(X\setminus L,\C).\end{equation}
		\end{thm}
		
		It has to be noted (Remark \ref{conjterao}) that according to this theorem, a conjecture of H. Terao~\cite{teraologarithmic} is equivalent to the fact that the inclusion~$\Omega^\bullet_{\langle X,L\rangle}\subset\Omega^\bullet_X(\log L)$ is a quasi-isomorphism.\\
		
		The proof of Theorem \ref{introqis} is local and relies on the Brieskorn-Orlik-Solomon theorem. Another central technical tool is the weight filtration~$W$ on~$\Omega^\bullet_{\langle X,L\rangle}$: we define~$W_k\Omega^\bullet_{\langle X,L\rangle}\subset \Omega^\bullet_{\langle X,L\rangle}$ to be the subcomplex spanned by the forms (\ref{eqlogformsncd1}) with~$s\leq k$. In view of the isomorphism (\ref{isoiso}), we get a filtration on the cohomology of~$X\setminus L$ which is proved to be defined over~$\Q$. Together with the Hodge filtration~$F^p\Omega^\bullet_{\langle X,L\rangle}=\Omega^{\geq p}_{\langle X,L\rangle}$, it defines a mixed Hodge structure on~$H^\bullet(X\setminus L)$. The functoriality of our construction then implies that this is the same as the mixed Hodge structure defined by Deligne.\\
		
		According to the general theory of mixed Hodge structures, the hypercohomology spectral sequence associated to the weight filtration degenerates at the~$E_2$-term, hence the~$E_1$-term gives a model for the cohomology of~$X\setminus L$. We then prove that this model is indeed the Orlik-Solomon model~$M^\bullet(X,L)$. This concludes the proof of Theorem \ref{intromaintheorem}.

	\subsection{Outline of this article}
	
		In \S$2$ we recall some classical facts about the Orlik-Solomon algebra and the Brieskorn-Orlik-Solomon theorem in the framework of  hyperplane arrangements, and introduce the Orlik-Solomon algebra of a hypersurface arrangement.\\
		In \S$3$, we introduce the complex of logarithmic forms along a  hyperplane arrangement and its weight filtration, and prove the local form (Theorem \ref{qis}) of the comparison theorem \ref{introqis}. Then we globalize our results to the framework of hypersurface arrangements (Theorem \ref{qisglobal}).\\
		In \S$4$, we use the formalism of mixed Hodge complexes to give an alternative definition of the mixed Hodge structure on the cohomology of~$X\setminus L$. This allows us to prove Theorem \ref{intromaintheorem} (Theorem \ref{maintheorem}).\\
		In \S$5$, we study the functoriality of the Orlik-Solomon model with respect to blow-ups, giving explicit formulas (Theorem \ref{formulaMpi}).\\
		In \S$6$, we apply our results to configuration spaces of points on curves and prove (Theorem \ref{compbloch}) the isomorphism between the Orlik-Solomon model and the model proposed by Kriz and Totaro and generalized by Bloch.

	\subsection{Conventions and notations}
		
		\begin{enumerate}
		\item \textit{(Coefficients)} Unless otherwise stated, all vector spaces, algebras, as well as tensor products of such objects, are implicitly defined over~$\Q$. All (mixed) Hodge structures are implicitly defined over~$\Q$.
		\item \textit{(Cohomology)} If~$Y$ is a complex manifold, we will simply write~$H^p(Y)$ for the~$p$-th singular cohomology group of~$Y$ with rational coefficients. We will write~$H^p(Y,\C)=H^p(Y)\otimes\C$ for the~$p$-th singular cohomology group of~$Y$ with complex coefficients. This group is naturally isomorphic, via the de Rham isomorphism, to the~$p$-th de Rham cohomology group of~$Y$ tensored with~$\C$, hence we allow ourselves to use smooth differential forms as representatives for cohomology classes.
		\end{enumerate}
		
	\subsection{Acknowledgements}	
		
		The author thanks Francis Brown for many corrections and comments on this article, Spencer Bloch for helpful discussions and for giving him a preliminary version of~\cite{blochtreeterated}, Eduard Looijenga for pointing out to him the reference~\cite{looijenga} which helped simplify the presentation, Christin Bibby, Alexandru Dimca and H\'{e}l\`{e}ne Esnault for useful comments on a preliminary version. This work was partially supported by ERC grant 257638 ``Periods in algebraic geometry and physics''.

\section{The Orlik-Solomon algebra of a hypersurface arrangement}
	
		We first recall some classical facts about  hyperplane arrangements. The interested reader will find more details in the expository book~\cite{orlikterao} or the survey~\cite{yuzvinskiorliksolomon}. Then we introduce hypersurface arrangements, define their Orlik-Solomon algebras and discuss their functoriality properties.
		 	
		\subsection{The Orlik-Solomon algebra of a hyperplane arrangement}\label{OScentral}
	
		A \emph{hyperplane arrangement} in~$\C^n$ is a finite set~$L$ of hyperplanes of~$\C^n$, all containing the origin.\footnote{In many references, this would be called a central hyperplane arrangement.} For a matter of notation, we will implicitly fix a linear ordering on the hyperplanes and write~$L=\{L_1,\ldots,L_l\}$. Nevertheless, the objects that we will define out of a  hyperplane arrangement will be independent of such an ordering. \\
		We will use the same letter~$L$ to denote the union of the hyperplanes:
		$$L=L_1\cup\cdots\cup L_l.$$
		
		For a subset~$I\subset\{1,\ldots, l\}$, the \emph{stratum} of the arrangement~$L$ indexed by~$I$ is the vector space~$L_I=\bigcap_{i\in I}L_i$ with the convention~$L_{\varnothing}=\C^n$. We write~$\s_\bullet(L)$ for the set of strata of~$L$, graded by the codimension, so that~$\s_0(L)=\{\C^n\}$ and~$\s_1(L)=\{L_1,\ldots,L_l\}$. With the order given by reverse inclusion,~$\s_\bullet(L)$ is given the structure of a graded poset, called the \emph{poset} of the  hyperplane arrangement~$L$.\\
		
		We set~$\Lambda_\bullet(L)=\Lambda^\bullet(e_1,\ldots, e_l)$, the exterior algebra over~$\Q$ with a generator~$e_i$ in degree~$1$ for each~$L_i$. Let~$\delta:\Lambda_\bullet(L)\rightarrow \Lambda_{\bullet-1}(L)$ be the unique derivation of~$\Lambda_\bullet(L)$ such that~$\delta(e_i)=1$ for~$i=1,\ldots,l$.\\		
		For~$I=\{i_1<\cdots<i_k\}\subset\{1,\ldots,l\}$ we set~$e_I=e_{i_1}\wedge \cdots\wedge e_{i_k}\in \Lambda_k(L)$ with the convention~$e_\varnothing=1$. The derivation~$\delta$ is then given by the formula
		$$\delta(e_I)=\sum_{s=1}^{k}(-1)^{s-1}e_{i_1}\wedge\cdots\wedge\widehat{e_{i_s}}\wedge\cdots\wedge e_{i_k}.$$
		
		A subset~$I\subset\{1,\ldots,l\}$ is said to be \emph{dependent} (resp. \emph{independent}) if~$\mathrm{codim}(L_I)<|I|$ (resp.~$\mathrm{codim}(L_I)=|I|$), which is equivalent to saying that the linear forms defining the~$L_i$'s, for~$i\in I$, are linearly dependent (resp. independent). Let~$J_\bullet(L)$ be the homogeneous ideal of~$\Lambda_\bullet(L)$ generated by the elements~$\delta(e_I)$ for~$I\subset\{1,\ldots,l\}$ dependent. The quotient
		$$A_\bullet(L)=\Lambda_\bullet(L)/J_\bullet(L)$$
		is a graded~$\Q$-algebra called the \emph{Orlik-Solomon algebra} of the  hyperplane arrangement~$L$. It only depends on the poset of~$L$. 
		
		
		For a stratum~$S$, let~$A_S(L)$ to be the sub-vector space of~$A_\bullet(L)$ spanned by the monomials~$e_I$ for~$I$ such that~$L_I=S$. One easily sees that we have a direct sum decomposition
		\begin{equation}\label{eqdirectsumS}
		A_\bullet(L)=\bigoplus_{S\in\s_\bullet(L)}A_S(L)
		\end{equation}
		and~$A_S(L)$ only depends on the  hyperplane arrangement consisting of the hyperplanes in~$L$ that contain~$S$, and more precisely on its poset.
		
		The product in~$A_\bullet(L)$ splits with respect to the direct sum decomposition (\ref{eqdirectsumS}), with components
		\begin{equation}\label{eqprodS}
		A_S(L)\otimes A_{S'}(L)\rightarrow A_{S\cap S'}(L)
		\end{equation}
		which are zero if~$\mathrm{codim}(S\cap S')<\mathrm{codim}(S)+\mathrm{codim}{S'}$.\\
		The derivation~$\delta$ induces a derivation~$\delta:A_{\bullet}(L)\rightarrow A_{\bullet-1}(L)$ which splits with respect to the direct sum decomposition (\ref{eqdirectsumS}), with components
		\begin{equation}\label{eqderivS}
		A_S(L)\rightarrow A_{S'}(L)
		\end{equation}
		for~$S\subset S'$,~$\mathrm{codim}(S')=\mathrm{codim}(S)-1$.
		
		\subsection{Deletion and restriction}\label{delres}
	
		Let~$L=\{L_1,\ldots,L_l\}$ be a  hyperplane arrangement in~$\C^n$ such that~$l\geq 1$. In this article we will only be concerned about deletion and restriction with respect to the last hyperplane~$L_l$. The \emph{deletion} of~$L$ (with respect to~$L_l$) is the arrangement~$L'=\{L_1,\ldots,L_{l-1}\}$ in~$\C^n$. The \emph{restriction} of~$L$ (with respect to~$L_l$) is the arrangement~$L''$ on~$L_l\cong \C^{n-1}$ consisting of all the intersections of~$L_l$ with the~$L_i$'s,~$i=1,\ldots,l-1$. If the hyperplanes~$L_i$ are not in general position, it may happen that the cardinality~$l''$ of~$L''$ is less than~$l-1$.\\
		
		For all~$k$, we have a short exact sequence of~$\Q$-vector spaces, called the \emph{deletion-restriction short exact sequence}, see~\cite[Theorem 3.65]{orlikterao} or~\cite[Corollary 2.17]{yuzvinskiorliksolomon}:
		\begin{equation}\label{eqdelres}
		0\rightarrow A_k(L')\stackrel{i}{\rightarrow} A_k(L) \stackrel{j}{\rightarrow} A_{k-1}(L'')\rightarrow 0.
		\end{equation}
		 This exact sequence splits with respect to the direct sum decomposition (\ref{eqdirectsumS}). For~$S$ a stratum of~$L$, there are three cases:
		\begin{itemize}
		\item ~$S$ is not contained in~$L_l$, then it is not a stratum of~$L''$ but is a stratum of~$L'$, and we just get an isomorphism
		$$0\rightarrow A_S(L')\rightarrow A_S(L)\rightarrow 0\rightarrow 0;$$
		\item~$S$ is contained in~$L_l$ but is not a stratum of~$L'$, and we just get an isomorphism
		$$0\rightarrow 0\rightarrow A_S(L)\rightarrow A_S(L'')\rightarrow 0;$$
		\item~$S$ is contained in~$L_l$ and is a stratum of~$L'$, and we get a short exact sequence
		$$0\rightarrow A_S(L')\rightarrow A_S(L)\rightarrow A_S(L'')\rightarrow 0.$$
		\end{itemize}
		
		\subsection{The Brieskorn-Orlik-Solomon theorem}\label{secBOS}
		
		Let~$L=\{L_1,\ldots,L_l\}$ be a  hyperplane arrangement in~$\C^n$. For~$i=1,\ldots,l$ we fix a linear form~$f_i$ on~$\C^n$ such that~$L_i=\{f_i=0\}$. Such a form is unique up to a non-zero multiplicative constant. We define holomorphic~$1$-forms on~$\C^n\setminus L$:~$$\omega_i=\dfrac{df_i}{f_i}\cdot$$
		For a subset~$I=\{i_1<\cdots<i_k\}\subset\{1,\ldots,l\}$ we set~$\omega_I=\omega_{i_1}\wedge\cdots\wedge\omega_{i_k}$.\\
		Let~$\Omega^\bullet(\C^n\setminus L)$ be the algebra of global holomorphic forms on~$\C^n\setminus L$ and~$R^\bullet(L)\subset \Omega^\bullet(\C^n\setminus L)$ be the subalgebra over~$\Q$ generated by~$1$ and the forms~$\frac{1}{2i\pi}\omega_i$ for~$i=1,\ldots,l$. We define a morphism of graded algebras~$u:\Lambda_\bullet(L)\rightarrow R^\bullet(L)$ by the formula~$$u(e_i)=\dfrac{1}{2i\pi}\omega_i.$$ 
		

		A simple computation shows that~$u$ passes to the quotient and defines a map of graded algebras~$$u:A_\bullet(L)\rightarrow R^\bullet(L).$$
		
		Each form~$\frac{1}{2i\pi}\omega_i$ is closed and its class is in the cohomology of~$\C^n\setminus L$ with \textit{rational} (and even integer) coefficients, thus there is a well-defined map of graded algebras~$$v:R^\bullet(L)\rightarrow H^\bullet(\C^n\setminus L).$$
		
		\begin{thm}[Brieskorn-Orlik-Solomon theorem]\label{bos}
		The maps~$u$ and~$v$ are isomorphisms of graded algebras:
		$$A_\bullet(L) \overset{\substack{u\\\simeq}}{\longrightarrow} R^\bullet(L) \overset{\substack{v\\\simeq}}{\longrightarrow} H^\bullet(\C^n\setminus L).$$
		\end{thm}
		
		\begin{rem}
		The fact that~$v$ is an isomorphism was first conjectured by Arnol'd~\cite{arnold} and then proved by Brieskorn~\cite{brieskorn}. The fact that~$u$ is an isomorphism was proved by Orlik and Solomon~\cite{orliksolomon}. A proof may be found in~\cite[Theorems 3.126 and 5.89]{orlikterao}.
		\end{rem}
		
	\subsection{The Orlik-Solomon algebra of a hypersurface arrangement}\label{OSglobal}
		
		We write~$\Delta=\{|z|<1\}\subset\C$ for the open unit disk and~$\Delta^n\subset\C^n$ for the unit~$n$-dimensional polydisk. Let~$X$ be a complex manifold. The following terminology is borrowed from P. Aluffi~\cite{aluffi}.
	
		\begin{defi}
		 A finite set~$L=\{L_1,\ldots,L_l\}$ of smooth hypersurfaces of~$X$ is a \emph{hypersurface arrangement} if around each point of~$X$ we may find a system of local coordinates in which each~$L_i$ is defined by a linear equation. In other words,~$X$ is covered by charts~$V\cong \Delta^n$ such that for all~$i$,~$L_i\cap V$ is the intersection of~$\Delta^n$ with a linear hyperplane in~$\C^n$.
		\end{defi}
		 
		As for  hyperplane arrangements, the objects that we will define out of a hypersurface arrangement will be independent of the linear ordering on the hypersurfaces~$L_i$. We use the same letter~$L$ to denote the union of the hypersurfaces:
		$$L=L_1\cup\cdots\cup L_l.$$
		
		The notion of hypersurface arrangement generalizes that of (simple) normal crossing divisor: a hypersurface arrangement is a normal crossing divisor if the local linear equations defining the~$L_i$'s are everywhere linearly independent, i.e. if we can always choose local coordinates such that the irreducible components~$L_i$ are coordinate hyperplanes.\\
		
		For a subset~$I\subset\{1,\ldots,l\}$, we still write~$L_I=\bigcap_{i\in I}L_i$, which is a disjoint union of complex submanifolds of~$X$. 
		A \emph{stratum} of~$L$ is a non-empty connected component of some~$L_I$; it is a complex submanifold of~$X$. We write~$\s_\bullet(L)$ for the set of strata of~$L$, graded by the codimension. We give~$\s_\bullet(L)$ the structure of a graded poset using reverse inclusion, and call it the \emph{poset} of the hypersurface arrangement~$L$.\\
		
		Let~$p$ be a point in~$X$ and~$V$ be a neighbourhood of~$p$. Then any chart~$V\cong\Delta^n$ as in the above definition defines a  hyperplane arrangement denoted~$L^{(p)}$ in~$\C^n$. It is an abuse of notation since choosing another chart gives a different  hyperplane arrangement, but it will not matter since we will only be interested in the poset of~$L^{(p)}$, which is well-defined. More intrinsically,~$L^{(p)}$ may be read off the tangent space of~$X$ at~$p$. Let~$S$ be a stratum of~$L$; since~$S$ is connected, the poset consisting of the strata of~$L^{(p)}$ that contain~$S$ is independent of the point~$p\in S$, and we may define~$$A_S(L)=A_S(L^{(p)})$$ for any choice of point~$p\in S$. Let us then define
		$$A_\bullet(L)=\bigoplus_{S\in\s_\bullet(L)}A_S(L).$$
		We now give~$A_\bullet(L)$ the structure of a graded algebra. The product
		\begin{equation}\label{eqprodglobal}
		A_S(L)\otimes A_{S'}(L)\rightarrow A_T(L)
		\end{equation}
		is non-zero only if~$T$ is a connected component of~$S\cap S'$ such that~$\mathrm{codim}(T)=\mathrm{codim}(S)+\mathrm{codim}(S')$, and is then given by (\ref{eqprodS}) by choosing any point~$p\in T$.
		
		The graded algebra~$A_\bullet(L)$ is called the\emph{ Orlik-Solomon algebra} of the hypersurface arrangement~$L$.
		
		For~$S\subset S'$ an inclusion of strata of~$L$ such that~$\mathrm{codim}(S')=\mathrm{codim}(S)-1$, we define
		\begin{equation}\label{eqderivglobal}
		A_S(L)\rightarrow A_{S'}(L)
		\end{equation}
		as in the local case (\ref{eqderivS}) by choosing any point~$p\in S$. One should note that in general the map~$A_\bullet(L)\rightarrow A_{\bullet-1}(L)$ induced by (\ref{eqderivglobal}) is not a derivation of the Orlik-Solomon algebra.
		
		\begin{rem}
		Let us assume that 
		\begin{equation}\label{connectedassumption}
		\textnormal{for all~$I$,~$L_I$ is connected.}
		\end{equation}
		The Orlik-Solomon algebra of~$L=\{L_1,\ldots,L_l\}$ thus has a presentation similar to that of a hyperplane arrangement. A subset~$I\subset\{1,\ldots,l\}$ is said to be \emph{null} if~$L_I=\varnothing$ and \emph{dependent} (resp. \emph{independent}) if~$L_I\neq \varnothing$ and~$\mathrm{codim}(L_I)<|I|$ (resp.~$\mathrm{codim}(L_I)=|I|$). Then~$A_\bullet(L)$ is the quotient of~$\Lambda^\bullet(e_1,\ldots,e_l)$ by the homogeneous ideal generated by the monomials~$e_I$ for~$I$ null and the elements~$\delta(e_I)$ for~$I$ dependent. In the case of a general hyperplane arrangement (the hyperplanes do not necessarily contain the origin), we recover the classical definition~\cite[Definition 3.45]{orlikterao}.
		Without the assumption (\ref{connectedassumption}), the Orlik-Solomon algebra may not even be generated in degree~$1$.
		\end{rem}
		
		\subsection{Functoriality of the Orlik-Solomon algebra}\label{secfunctoriality}
		
		Let~$L=\{L_1,\ldots,L_l\}$ and~$L'=\{L'_1,\ldots,L'_{l'}\}$ be  hyperplane arrangements respectively in~$\C^n$ and~$\C^{n'}$. Let~$\varphi:\Delta^n\rightarrow\Delta^{n'}$ be a holomorphic map such that~$\varphi^{-1}(L')\subset L$, i.e.~$\varphi(\Delta^n\setminus L)\subset\Delta^{n'}\setminus L'$.\\
		Then~$\varphi$ induces a map~$\varphi^*:H^\bullet(\Delta^{n'}\setminus L')\rightarrow H^\bullet(\Delta^n\setminus L)$ in cohomology. The inclusions~$\Delta^n\setminus L\subset \C^n\setminus L$ and~$\Delta^{n'}\setminus L'\subset \C^{n'}\setminus L'$ are retractions and hence induce isomorphisms in cohomology. Thus the Brieskorn-Orlik-Solomon theorem \ref{bos} implies that there is a unique map of graded algebras 
		$$A_\bullet(\varphi):A_\bullet(L')\rightarrow A_\bullet(L)$$ that fits into the following commutative square.
		
		\begin{displaymath}
		\xymatrix{
		A_\bullet(L') \ar[r]^{A_\bullet(\varphi)} \ar[d]^{\cong}& A_\bullet(L) \ar[d]^{\cong}\\
		H^\bullet(\Delta^{n'}\setminus L') \ar[r]^{\varphi^*} & H^\bullet(\Delta^n\setminus L)
		}
		\end{displaymath}
		
		For~$j=1,\ldots,l'$, there is an equality
		$$f'_j\circ\varphi=u_j\prod_i f_i^{m_{ij}}$$
		between germs at~$0$ of holomorphic functions on~$\Delta^n$, with~$u_j$ a holomorphic function such that~$u_j(0)\neq 0$ and~$m_{ij}\geq 0$. One then sees that~$A_\bullet(\varphi):A_\bullet(L')\rightarrow A_\bullet(L)$ is the unique map of graded algebras such that for~$j=1,\ldots,l'$,~$$A_1(\varphi)(e'_j)=\sum_i m_{ij}e_i.$$
	
		We may globalize this construction; if~$L$ (resp.~$L'$) is a hypersurface arrangement in a complex manifold~$X$ (resp.~$X'$), and~$\varphi:X\rightarrow X'$ a holomorphic map such that~$\varphi^{-1}(L')\subset L$, then we define
		\begin{equation}\label{functorialityS}
		A_{S,S'}(\varphi):A_{S'}(L')\rightarrow A_S(L)
		\end{equation}
		for strata~$S\in\s_\bullet(L)$ and~$S'\in\s_\bullet(L')$ by looking at~$\varphi$ in local charts and applying the above definition. It is clear that this defines a map of graded algebras~$A_\bullet(\varphi):A_\bullet(L)\rightarrow A_\bullet(L')$ that is functorial in the sense that we have~$A_\bullet(\psi\circ\varphi)=A_\bullet(\varphi)\circ A_\bullet(\psi)$ whenever this is meaningful. If~$\varphi:X\rightarrow X\times X$ is the diagonal of~$X$, then~$A_\bullet(\varphi)$ is the product morphism~$A_\bullet(L)\otimes A_\bullet(L)\rightarrow A_\bullet(L)$.

\section{Logarithmic forms and the weight filtration}

		We define and study the forms with logarithmic poles along a  hyperplane arrangement. In \S \ref{seclog}, \ref{secres}, \ref{weight}, \ref{seccomp}, we focus on  hyperplane arrangements (the local case). The main results are Theorem \ref{gr} which computes its graded pieces, and Theorem \ref{qis} which states that the logarithmic complex computes the cohomology of the complement of the hyperplane arrangement. Then in \S \ref{secglobalforms} we extend our constructions and results to the case of hypersurface arrangements (the global case).
		
		If~$Y$ is a complex manifold, we write~$\Omega^p_Y$ for the sheaf of holomorphic~$p$-forms on~$Y$ and~$\Omega^p(Y)=\Gamma(Y,\Omega^p_Y)$ for the vector space of global holomorphic~$p$-forms on~$Y$.

		\subsection{The logarithmic complex}\label{seclog}
		
		Let~$L=\{L_1,\ldots,L_l\}$ be a  hyperplane arrangement in~$\C^n$. We recall that we defined some differential forms 
		$\omega_i=\frac{df_i}{f_i}$ for~$i=1,\ldots,l$, and~$\omega_I=\omega_{i_1}\wedge\cdots\wedge\omega_{i_k}$ for~$I=\{i_1<\cdots<i_k\}$, which is zero if~$I$ is dependent.
		 
		\begin{defi}
		A meromorphic form on~$\C^n$ is said to have \emph{logarithmic poles along~$L$} if it is a linear combination over~$\C$ of forms of the type
		$\eta\wedge\omega_I$ for some~$I\subset\{1,\ldots,l\}$, where~$\eta$ is a holomorphic form on~$\C^n$.
		\end{defi}
		
		We define~$\Omega^p\langle L\rangle$ to be the~$\C$-vector space of meromorphic~$p$-forms on~$\C^n$ with logarithmic poles along~$L$. These forms are stable under the exterior differential, hence we get a complex~$\Omega^\bullet\langle L\rangle$ that embeds into the complex of holomorphic forms on~$\C^n\setminus L$:		
		$$\Omega^\bullet\langle L\rangle \hookrightarrow \Omega^\bullet(\C^n\setminus L)$$
		which we call the \emph{complex of logarithmic forms} of~$L$.
		
		\begin{rem}\label{logcomplex}
		This definition is not standard in the theory of hyperplane arrangements. In~\cite{orlikterao}, following Saito~\cite{saitologarithmic}, one defines a complex~$\Omega^\bullet(\log L)$ in the following way. Let~$Q=f_1\cdots f_l$ be a defining polynomial for the arrangement. Then~$\Omega^p(\log L)$ is the set of meromorphic~$p$-forms~$\omega$ on~$\C^n$ such that~$Q\omega$ and~$Qd\omega$ are holomorphic.\\
		We have an inclusion~$\Omega^\bullet\langle L\rangle\subset\Omega^\bullet(\log L)$ which is an equality if and only if~$L=\{L_1,\ldots,L_l\}$ is independent. For instance, in~$\C^2$ with coordinates~$x$ and~$y$, let us look at~$L_1=\{x=0\}$,~$L_2=\{y=0\}$,~$L_3=\{x=y\}$. Then~$Q=xy(x-y)$ and the closed form~$\omega=\frac{dx\wedge dy}{xy(x-y)}$ is in~$\Omega^2(\log L)$ but not in~$\Omega^2\langle L\rangle$.
		\end{rem}
		
		\subsection{Residues}\label{secres}
		
		We briefly recall the notion of residue of a form with logarithmic poles along a  hyperplane arrangement. In the case of dimension~$n=1$, this is the usual Cauchy residue in complex analysis; the general notion of residue is due to Poincar\'{e} and Leray~\cite{leray}. For residues in the setting of hyperplane arrangements, see~\cite[3.124]{orlikterao}.\\
		
		We fix a  hyperplane arrangement~$L=\{L_1,\ldots,L_l\}$ in~$\C^n$. Let~$L'$ (resp.~$L''$) the deletion (resp. the restriction) of~$L$ with respect to~$L_l=\{f_l=0\}$. Let~$\omega$ be a~$p$-form on~$\C^n$ with logarithmic poles along~$L$. Then there exists a~$(p-1)$-form~$\alpha$ and a~$p$-form~$\beta$, both of which have logarithmic poles along~$L'$, such that
		$$\omega=\alpha\wedge\omega_l+\beta.$$ 
		The form~$$\mathrm{Res}_{L_l}(\omega)=2i\pi\,\alpha_{|L_l}$$ is independent of the choices. It is a~$(p-1)$-form on~$L_l$ with logarithmic poles along~$L''$, called the \emph{residue} of~$\omega$ along~$L_l$. We then have a morphism of complexes
		$$\mathrm{Res}_{L_l}:\Omega^\bullet\langle L \rangle \rightarrow \Omega^{\bullet-1}\langle L''\rangle$$
		where~$L''$ is the restriction of~$L$ with respect to~$L_l$. We then have a sequence of morphisms of complexes
		$$(\mathcal{R}): \hspace{.2cm} 0\rightarrow \Omega^\bullet\langle L'\rangle\stackrel{i}{\rightarrow} \Omega^\bullet\langle L\rangle \overset{\mathrm{Res}_{L_l}}{\longrightarrow} \Omega^{\bullet-1}\langle L''\rangle\rightarrow 0$$
		where~$i$ is the natural inclusion. It is obvious from the definitions that~$\mathrm{Res}_{L_l}\circ i=0$, that~$i$ is injective and~$\mathrm{Res}_{L_l}$ is surjective. We will prove in the next paragraph that~$\mathrm{ker}(\mathrm{Res}_{L_l})\subset \mathrm{Im}(i)$, so that the above sequence is a short exact sequence.
		
		\begin{rem}\label{iteratedresidues}
		When taking iterated residues, one should note that they ``do not commute'' in general, even when this has a clear meaning. For example, if~$L_1=\{x=0\}$,~$L_2=\{y=0\}$,~$L_3=\{x=y\}$ in~$\C^2$ and~$\omega=\frac{dx}{x}\wedge\frac{dy}{y}\in \Omega^2\langle L\rangle$, we have~$\mathrm{Res}_{L_2\cap L_3}\mathrm{Res}_{L_2}(\omega)=(2i\pi)^2$ and~$\mathrm{Res}_{L_3\cap L_2}\mathrm{Res}_{L_3}(\omega)=0$.
		\end{rem}
		
		\subsection{The weight filtration}\label{weight}
		
		We fix a  hyperplane arrangement~$L=\{L_1,\ldots,L_l\}$ in~$\C^n$. The following terminology is borrowed from P. Deligne~\cite[3.1.5]{delignehodge2}.
		
		\begin{defi}
		For~$k\geq 0$, we define~$W_k\Omega^\bullet\langle L\rangle \subset \Omega^\bullet\langle L\rangle$ to be the subcomplex spanned by the forms that are of the type ~$\eta\wedge\omega_I$ with~$|I|\leq k$, where~$\eta$ is a holomorphic form on~$\C^n$.
		These subcomplexes define an ascending filtration
		$$W_0\Omega^\bullet\langle L\rangle \subset W_1\Omega^\bullet\langle L\rangle \subset \cdots$$ 
		on~$\Omega^\bullet\langle L\rangle$ called the \emph{weight filtration}.
		\end{defi}
		
		We have~$W_0\Omega^\bullet\langle L\rangle=\Omega^\bullet(\C^n)$ and~$W_p\Omega^p\langle L\rangle=\Omega^p\langle L\rangle$.\\
		
		By definition, the residue morphisms induce morphisms~$\mathrm{Res}_{L_l}:W_k\Omega^\bullet\langle L\rangle\rightarrow W_{k-1}\Omega^{\bullet-1}\langle L''\rangle$ which are easily seen to be surjective. Thus the sequence~$(\mathcal{R})$ induces sequences
		\begin{equation}\label{WR}(W_k\mathcal{R}): \hspace{.2cm} 0\rightarrow W_k\Omega^\bullet\langle L'\rangle\stackrel{i}{\rightarrow} W_k\Omega^\bullet\langle L\rangle \overset{\mathrm{Res}_{L_l}}{\longrightarrow} W_{k-1}\Omega^{\bullet-1}\langle L''\rangle\rightarrow 0\end{equation} and
		\begin{equation}\label{grR}(\gr_k^W\mathcal{R}): \hspace{.2cm}0\rightarrow \gr_k^W\Omega^\bullet\langle L'\rangle\stackrel{i}{\rightarrow} \gr_k^W\Omega^\bullet\langle L\rangle \overset{\mathrm{Res}_{L_l}}{\longrightarrow} \gr_{k-1}^W\Omega^{\bullet-1}\langle L''\rangle\rightarrow 0.\end{equation}
		We will prove that they are short exact sequences. For now, the only easy facts are that~$(W_k\mathcal{R})$ is exact on the left and on the right, and that~$(\gr_k^W\mathcal{R})$ is exact on the right.\\
		
		The following lemma is easily proved by choosing appropriate coordinates on~$\C^n$.
		
		\begin{lem}\label{weightres}
		Let~$I\subset\{1,\ldots,l\}$,~$|I|=k$, be an independent subset and~$\eta$ a holomorphic form on~$\C^n$. If~$\eta_{|L_I}=0$ then~$\eta\wedge\omega_I\in W_{k-1}\Omega^\bullet\langle L\rangle$.
		\end{lem}
	
		For all~$k$, we define
		$$G_k^\bullet(L)=\bigoplus_{S\in\s_k(L)}\Omega^{\bullet-k}(S)\otimes A_S(L).$$ 
		
		This is a complex of~$\C$-vector spaces. We define a morphism of complexes~$$\Phi:G_k^\bullet(L)\rightarrow \gr_k^W\Omega^\bullet\langle L\rangle$$ in the following way. For~$I$ independent of cardinality~$k$, for~$\eta\in\Omega^{\bullet-k}(L_I)$, we set
		$$\Phi(\eta\otimes e_I)=(2i\pi)^{-k}\td{\eta}\wedge\omega_I$$
		where~$\td{\eta}\in\Omega^{\bullet-k}(\C^n)$ is any form such that~$\td{\eta}_{|L_l}=\eta$. 
		Lemma \ref{weightres} implies that this does not depend on the choice of~$\td{\eta}$ and one immediately sees that it passes to the quotient that defines the groups~$A_S(L)$. It is then easy to check that~$\Phi$ is a morphism of complexes.
		
		\begin{thm}\label{gr}
		The morphism~$\Phi:G_k^\bullet(L)\rightarrow \gr_k^W\Omega^\bullet\langle L\rangle$ is an isomorphism of complexes.
		\end{thm}
		
		\begin{proof}
		 The surjectivity is trivial; we prove the injectivity by induction on the cardinal~$l$ of the arrangement.\\
		 For~$l=0$, the only non-trivial case is~$k=0$ and~$\Phi$ is just the identity of~$\Omega^\bullet(\C^n)$.\\
		 Suppose that the statement is proved for arrangements of cardinality~$\leq l-1$ and take an arrangement~$L$ of cardinality~$l$. Tensoring the deletion-restriction short exact sequence from \S\ref{delres} with the complexes~$\Omega^{\bullet-k}(S)$ we get a short exact sequence of complexes of~$\C$-vector spaces
		~$$0\rightarrow G_k^\bullet(L')\rightarrow G_k^\bullet(L)\rightarrow G_{k-1}^{\bullet-1}(L'')\rightarrow 0.$$
		 We then have a diagram
		 \begin{displaymath}
		\xymatrix{
		0 \ar[r]& G_k^\bullet(L') \ar[r]\ar[d]^{\Phi} & G_k^\bullet(L) \ar[r]\ar[d]^{\Phi} & G_{k-1}^{\bullet-1}(L'') \ar[r]\ar[d]^{\Phi} & 0 \\
		0 \ar[r]& \gr_k^W\Omega^\bullet\langle L'\rangle \ar[r] & \gr_k^W\Omega^\bullet\langle L\rangle \ar[r] & \gr_{k-1}^W\Omega^{\bullet-1}\langle L''\rangle \ar[r]& 0 \\
		}
		\end{displaymath}
		where the bottom row is the sequence (\ref{grR}). This diagram is easily seen to be commutative.\\
		By the inductive hypothesis, the vertical arrows on the right and on the left are isomorphisms. Thus a diagram chase shows that the bottom row is exact in the middle.\\
		Now the complexes (\ref{WR}) and (\ref{grR}) give rise to a short exact sequence of complexes~$$0\rightarrow (W_{k-1}\mathcal{R})\rightarrow (W_k\mathcal{R}) \rightarrow (\gr_{k}^W\mathcal{R})\rightarrow 0.$$
		The long exact sequence in cohomology tells us that if~$(W_{k-1}\mathcal{R})$ is exact in the middle then it is also the case for~$(W_k\mathcal{R})$. Since~$(W_0\mathcal{R})$ is just the sequence
		$$0\rightarrow \Omega^\bullet(\C^n)\overset{\mathrm{id}}{\rightarrow} \Omega^\bullet(\C^n)\rightarrow 0\rightarrow 0,$$
		we show by induction on~$k$ shows that~$(W_k\mathcal{R})$ is exact in the middle, hence a short exact sequence, for all~$k$. Again, the long exact sequence in cohomology shows that~$(\gr_k^W\mathcal{R})$ is also a short exact sequence for all~$k$.\\
		Thus, in the above commutative diagram, both rows are exact and a diagram chase (the~$5$-lemma) shows that the middle~$\Phi$ is injective. This completes the induction and the proof of the theorem.
		 \end{proof}
		 
		\begin{rem}\label{rempsi}
		The inverse morphism~$\Psi:\gr_k^W\Omega^\bullet\langle L\rangle \rightarrow G_k^\bullet(L)$ is given, for~$\eta$ holomorphic and~$I$ independent of cardinality~$k$, by~$$\Psi(\eta\wedge\omega_I)=(2i\pi)^k\,\eta_{|L_I}\in \Omega^{\bullet-k}(L_I)$$
		For~$k=1$ this is exactly the definition of a residue, but for~$k>1$ one should note that this has nothing to do with an ``iterated residue'' (see Remark \ref{iteratedresidues}).
		\end{rem}
		
		Since~$(\mathcal{R})=(W_k\mathcal{R})$ for~$k$ large enough, the proof of Theorem \ref{gr} implies the following.
		
		\begin{thm}\label{resexact}
		The sequences~$(\mathcal{R})$,~$(W_k\mathcal{R})$ and~$(\gr_k^W\mathcal{R})$ are short exact sequences of complexes.
		\end{thm}

		\subsection{The comparison theorem}\label{seccomp}
		
		\begin{thm}\label{qis}
		The inclusion~$\Omega^\bullet\langle L\rangle \hookrightarrow \Omega^\bullet(\C^n\setminus L)$ is a quasi-isomorphism.
		\end{thm}
				
		\begin{proof}
		Since~$\C^n\setminus L$ is a smooth affine algebraic variety over~$\C$, the cohomology of~$\Omega^\bullet(\C^n\setminus L)$ is the cohomology of~$\C^n\setminus L$ with complex coefficients. Thus we have to prove that the natural map
		$$H^p(\Omega^\bullet\langle L\rangle)\rightarrow H^p(\C^n\setminus L,\C)$$ is an isomorphism for all~$p$. We proceed by induction on the cardinality~$l$ of the arrangement. For~$l=0$ the statement is trivial. To pass from~$l-1$ to~$l$ we consider the commutative diagram
		\begin{displaymath}
		\xymatrix{
		\cdots \ar[r]&H^p(\Omega^\bullet\langle L'\rangle) \ar[r]\ar[d] & H^p(\Omega^\bullet\langle L\rangle) \ar[r]\ar[d] & H^{p-1}(\Omega^\bullet\langle L''\rangle) \ar[r]\ar[d] & \cdots \\
		0 \ar[r]& H^p(\C^n\setminus L') \ar[r] & H^p(\C^n\setminus L) \ar[r] & H^{p-1}(L_l\setminus L'') \ar[r]& 0 \\
		}
		\end{displaymath}
		The first row is the long exact sequence in cohomology associated to~$(\mathcal{R})$, the second row is induced by the deletion-restriction exact sequence via the Brieskorn-Orlik-Solomon theorem. Both rows are exact. By induction the vertical arrows on the left and on the right are isomorphisms. A classical diagram chase implies that the vertical arrow in the middle is also an isomorphism.
		\end{proof}
		
		\begin{rem}\label{conjterao}
		We have the inclusions of complexes
		$$\Omega^\bullet\langle L\rangle \overset{i_1}{\hookrightarrow} \Omega^\bullet(\log L) \overset{i_2}{\hookrightarrow} \Omega^\bullet(\C^n\setminus L)$$
		where~$\Omega^\bullet(\log L)$ has been defined in Remark \ref{logcomplex}.\\
		A conjecture by H. Terao~\cite{teraologarithmic} states that~$i_2$ is a quasi-isomorphism. According to Theorem \ref{qis}, the composite~$i_2\circ i_1$ is a quasi-isomorphism, hence Terao's conjecture is equivalent to the fact that~$i_1$ is a quasi-isomorphism. This is equivalent to the acyclicity of the quotient complex~$\Omega^\bullet(\log L)/\Omega^\bullet\langle L\rangle$.
		\end{rem}

	\subsection{Logarithmic forms along hypersurface arrangements}\label{secglobalforms}

		In this paragraph we globalize the definitions of the logarithmic complex and the weight filtration. As in the local case, we determine the weight-graded parts of the logarithmic complex and prove a comparison theorem. This generalizes the case of normal crossing divisors, studied by Deligne in~\cite[3.1]{delignehodge2}.\\
	
		Let~$X$ be a complex manifold and~$L$ a hypersurface arrangement in~$X$. A meromorphic form on~$X$ is said to have \emph{logarithmic poles along~$L$} if it is locally a linear combination over~$\C$ of forms of the type
		\begin{equation}
		\label{localform}\eta\wedge\frac{df_{i_1}}{f_{i_1}}\wedge\cdots\wedge\frac{df_{i_r}}{f_{i_r}}
		\end{equation}
		 with~$\eta$ holomorphic and the~$f_i$'s local defining (linear) equations for the~$L_i$'s. The meromorphic forms on~$X$ with logarithmic poles along~$L$ form a complex of sheaves of~$\C$-vector spaces on~$X$, that we denote by~$\Omega^\bullet_{\langle X,L\rangle}$. As in the local setting (Remark \ref{logcomplex}), we should point out that~$\Omega^\bullet_{\langle X,D\rangle}$ differs from Saito's complex~$\Omega^\bullet_X(\log L)$ if~$L$ is not a normal crossing divisor.
		
		It was pointed out to us by A. Dimca that the sheaves~$\Omega^1_{\langle X,L\rangle}$ have been previously defined in~\cite{catanesehostenkhetansturmfels} (where they are denoted~$\Omega_X(\log L)$) and~\cite{dolgachevlogarithmic} (where they are denoted~$\tilde{\Omega}_X(\log L)$).\\
		 
		 We globalize the \emph{weight filtration} on~$\Omega^\bullet_{\langle X,L\rangle}$ to get subcomplexes of sheaves~$W_k\Omega^\bullet_{\langle X,L\rangle}\subset\Omega^\bullet_{\langle X,L\rangle}$.\\
		
		The complex of sheaves~$\Omega^\bullet_{\langle X,L\rangle}$ is functorial in~$(X,L)$ in the following sense. If~$L'$ is another hypersurface arrangement in a complex manifold~$X'$, and if we have a holomorphic map~$\varphi:X\rightarrow X'$ such that~$\varphi^{-1}(L')\subset L$, then there is a pull-back map
		$$\varphi^*:\varphi^{-1}\Omega^\bullet_{\langle X',L' \rangle}\rightarrow \Omega^\bullet_{\langle X,L \rangle}$$
		that is compatible with composition in the usual sense. This follows from the discussion in \S\ref{secfunctoriality}. The weight filtration is also functorial.\\

		For a stratum~$S$ we denote by~$i_S:S\hookrightarrow X$ the closed immersion of~$S$ inside~$X$. We globalize the definition of~$G_k^\bullet(L)$ from \S\ref{weight} and define a complex of sheaves of~$\C$-vector spaces on~$X$:
		$$\mathcal{G}^\bullet_k(X,L)=\bigoplus_{S\in\s_k(L)} (i_S)_*\Omega^{\bullet-k}_{S} \otimes A_S(L).$$
		
		As in the local case, we may define a morphism of complexes of sheaves
		$$\Phi:\mathcal{G}^\bullet_k(X,L)\rightarrow \gr_k^W\Omega^\bullet_{\langle X,L\rangle}$$
		by putting
		$$\Phi(\eta\otimes e_I)=(2i\pi)^{-k}\,\td{\eta}\wedge\frac{df_{i_1}}{f_{i_1}}\wedge\cdots\wedge\frac{df_{i_k}}{f_{i_k}}$$
		for~$I=\{i_1<\cdots<i_k\}$,~$\eta\in\Omega^{\bullet-k}_S$ a local section,~$\td{\eta}\in\Omega^{\bullet-k}_X$ a local extension of~$\eta$, and the~$f_i$'s local equations for the~$L_i$'s. This definition is independent from the choice of the local equations~$f_i$. The following theorem is a global version of Theorem \ref{gr}. 
		
		\begin{thm}\label{grglobal}
		The morphism~$\Phi:\mathcal{G}^\bullet_k(X,L)\rightarrow \gr_k^W\Omega^\bullet_{\langle X,L\rangle}$ is an isomorphism.
		\end{thm}
		
		\begin{proof}
		It is enough to prove that for every chart~$V\cong\Delta^n$ on which~$L$ is a  hyperplane arrangement, the morphism
		$$\Gamma(V,\mathcal{G}^\bullet_k(X,L))\rightarrow \Gamma(V,\gr_k^W\Omega^\bullet_{\langle X,L\rangle})$$ is an isomorphism. This is exactly Theorem \ref{gr} with the ambient space~$\C^n$ replaced by the polydisk~$\Delta^n$. One can check that the proof of Theorem \ref{gr} can be copied word for word in that local setting.
		\end{proof}
		
		\begin{rem}
		The inverse morphism~$\Psi: \gr_k^W\Omega^\bullet_{\langle X,L\rangle}\rightarrow\mathcal{G}^\bullet_k(X,L)$ is given locally by the same formula as in Remark \ref{rempsi}. As already noted, this should not be mistaken with an iterated residue, unless~$L$ is a normal crossing divisor (in this case, Deligne calls~$\Psi$ the Poincar\'{e} residue, see~\cite[3.1.5.2]{delignehodge2}).
		\end{rem}

		Let~$j:X\setminus L\hookrightarrow X$ be the open immersion of the complement of~$L$ inside~$X$. The following theorem is a global version of Theorem \ref{qis}.
		
		\begin{thm}\label{qisglobal}
		The inclusion~$\Omega^\bullet_{\langle X,L\rangle}\hookrightarrow j_*\Omega^\bullet_{X\setminus L}$ is a quasi-isomorphism.
		\end{thm}
		
		\begin{proof}
		It is enough to prove that for every chart~$V\cong\Delta^n$ on which~$L$ is a  hyperplane arrangement, the morphism
		$$\Gamma(V,\Omega^\bullet_{\langle X,L\rangle}) \rightarrow \Gamma(V,j_*\Omega^\bullet_{X\setminus L})=\Omega^\bullet(V\setminus  L)$$
		is a quasi-isomorphism. This is exactly Theorem \ref{qis} with the ambient space~$\C^n$ replaced by the polydisk~$\Delta^n$. One can check that the proof of Theorem \ref{qis} can be copied word for word in the local setting. The argument that the strata~$L_I$ are contractible has to be replaced by the fact that the local strata~$\Delta^n\cap L_I$ are contractible (because they are polydisks). The Brieskorn-Orlik-Solomon theorem remains true in the local setting because the inclusion~$\Delta^n\setminus L\subset\C^n\setminus L$ is a retraction and hence induces an isomorphism in cohomology.
		\end{proof}
		
\section{A functorial mixed Hodge structure and the Orlik-Solomon model}

	If~$X$ is a smooth projective variety and~$L$ is a hypersurface arrangement in~$X$, we put a functorial mixed Hodge structure on the cohomology of the complement~$X\setminus L$. Our construction mimicks Deligne's~\cite{delignehodge2} in the case of normal crossing divisors.
		
	\subsection{Mixed Hodge complexes}
		
		We refer to~\cite[7.1,8.1]{delignehodge3} for the definitions of mixed Hodge complexes. If~$\mathbb{K}$ is a field, the filtered (resp. bifiltered) derived category of (bounded from above) complexes of~$\mathbb{K}$-vector spaces on a complex manifold~$Y$ is denoted by~$\mathrm{D^+F}(Y,\mathbb{K})$ (resp.~$\mathrm{D^+F}_2(Y,\mathbb{K})$). A cohomological mixed Hodge complex on~$Y$ is a triple 
		$$\K=((\K_\Q,W),(\K_\C,W,F),\alpha)$$
		with~$(\K_\Q,W)\in \mathrm{D^+F}(Y,\Q)$,~$(\K_\C,W,F)\in \mathrm{D^+F}_2(Y,\C)$ and~$\alpha:(\K_\Q,W)\otimes\C\cong(\K_\C,W)$ an isomorphism in~$\mathrm{D^+F}(Y,\C)$. These data must satisfy some compatibility conditions.
		
	 	The following theorem~\cite[8.1.9]{delignehodge3} is the fundamental theorem of mixed Hodge complexes. Our convention for spectral sequences uses decreasing filtrations. One passes from an increasing filtration~$\{W_p\}_{p\in\mathbb{Z}}$ to a decreasing filtration~$\{W^p\}_{p\in\mathbb{Z}}$ by putting~$W^p=W_{-p}$.
		
		\begin{thm}\label{mhc}
		Let~$Y$ be a complex manifold and $\K=((\K_\Q,W),(\K_\C,W,F),\alpha)$ be a cohomological mixed Hodge complex on~$Y$.
		\begin{enumerate}
		 \item For all~$n$, the filtration~$W[-n]$ and the filtration~$F$ define a mixed Hodge structure on~$\mathbb{H}^n(\K_\Q)$.
		 \item Let~$\wE$ be the cohomological spectral sequence defined by~$(\K_\Q,W)$. Then for all~$(p,q)$, the filtration~$F$ induces on~$\wE_1^{-p,q}=\mathbb{H}^{-p+q}(\gr_p^W\K_\Q)$ a Hodge structure of weight~$q$ and the differentials~$d_1^{-p,q}$ are morphisms of Hodge structures.
		 \item The spectral sequence~$\wE$ degenerates at~$E_2$:~$\wE_2^{-p,q}=\wE_\infty^{-p,q}=\gr_{p}^W\mathbb{H}^n(\K_\Q)=\gr_q^{W[-n]}\mathbb{H}^n(\K_\Q)$ for~$n=-p+q$.
		\end{enumerate}
		\end{thm}
		
	\subsection{A functorial mixed Hodge structure}
		
		Let~$X$ be a smooth projective variety over~$\C$ and~$L$ a hypersurface arrangement in~$X$. We use the previous constructions to put a functorial mixed Hodge structure on the cohomology~$H^\bullet(X\setminus L)$ of the complement, using the formalism of mixed Hodge complexes. This generalizes the case of normal crossing divisors, studied by Deligne in~\cite[3.2]{delignehodge2}, and summarized in terms of mixed Hodge complexes in~\cite[8.1.8]{delignehodge3}. We recall the notation~$j:X\setminus L\hookrightarrow X$.\\
		
		We define a triple
		$$\K(X,L)=((\K_\Q(X,L),W),(\K_\C(X,L),W,F),\alpha)$$ 
		in the following way:
		\begin{enumerate}
		\item~$\K_\Q(X,L)=Rj_*\Q_{X\setminus L}$ with the filtration~$W=\tau$, the canonical filtration~\cite[1.4.6]{delignehodge2}.
		\item~$\K_\C(X,L)=\Omega^\bullet_{\langle X,L\rangle}$ with the weight filtration~$W$ defined in \S\ref{secglobalforms}, and the Hodge filtration~$F$ defined by~$$F^p\Omega^\bullet_{\langle X,L\rangle}=\Omega^{\geq p}_{\langle X,L\rangle}.$$
		\item We have isomorphisms in~$\mathrm{D^+}(X,\C)$:~$$Rj_*\Q_{X\setminus L}\otimes\C\cong Rj_*\C_{X\setminus L}\cong j_*\Omega^\bullet_{X\setminus L}\cong\Omega^\bullet_{\langle X,L\rangle}$$ the last one being the quasi-isomorphism of the comparison theorem \ref{qisglobal}.\\ Hence we have an isomorphism~$(Rj_*\Q_{X\setminus L}\otimes\C,\tau)\cong(\Omega^\bullet_{\langle X,L\rangle},\tau)$ in~$\mathrm{D^+F}(X,\C)$. Finally the identity gives a filtered quasi-isomorphism~$(\Omega^\bullet_{\langle X,L\rangle},\tau)\cong(\Omega^\bullet_{\langle X,L\rangle},W)$, as follows from the same proof as in~\cite[3.1.8]{delignehodge2}, in view of the comparison theorem \ref{qisglobal}. This gives the isomorphism 
		$$\alpha:(Rj_*\Q_{X\setminus L},\tau)\otimes\C\cong(\Omega^\bullet_{\langle X,L\rangle},W)$$ in~$\mathrm{D^+F}(X,\C)$.
		\end{enumerate}
		
		\begin{thm}\label{mhs}
		The triple~$\K(X,L)$ is a cohomological mixed Hodge complex on~$X$, which is functorial with respect to the pair~$(X,L)$. It thus defines a functorial mixed Hodge structure on~$\mathbb{H}^n(Rj_*\Q_{X\setminus L})\cong H^n(X\setminus L)$ for all~$n$.
		\end{thm}
		
		Here, functoriality has to be understood in the sense of \S\ref{secfunctoriality}.		
		
		
		\begin{proof}
		Theorem \ref{gr} gives an isomorphism 
		$$\gr_k^W\Omega^\bullet_{\langle X,L\rangle}\cong\bigoplus_{S\in\s_k(L)}(i_S)_*\Omega^{\bullet-k}_S \otimes A_S(L).$$
		A local computation as in~\cite[Lemma 4.9]{peterssteenbrink}, shows that this isomorphism is defined over~$\Q$ if we take care of the Tate twists. In other words we have a commutative diagram:
		$$
		\xymatrix{
		\gr_k^W\Omega^\bullet_{\langle X,L\rangle} \ar[r]^-{\cong} & \bigoplus_{S\in\s_k(L)}(i_S)_*\Omega^{\bullet}_S [-k]\otimes A_S(L) \\
		\gr_k^\tau Rj_*\C_U \ar[r]^-{\cong} \ar[u]^{\cong} & \bigoplus_{S\in\s_k(L)}(i_S)_*\C_S  [-k]\otimes A_S(L) \ar[u]^{\cong}\\
		\gr_k^\tau Rj_*\Q_U \ar[r]^-{\cong} \ar[u] &  \bigoplus_{S\in\s_k(L)}(i_S)_*\Q_S [-k](-k)\otimes A_S(L) \ar[u]\\
		}
		$$
		To complete the proof it is enough to notice that the top row of this diagram is compatible with the Hodge filtrations. Hence we get
		$$\gr_k^W\K(X,L)=\bigoplus_{S\in\s_k(L)} (i_S)_*\K(S)[-k](-k)\otimes A_S(L)$$ which is a cohomological Hodge complex of weight~$k$. \\
		The functoriality statement follows from the functoriality of the sheaves of logarithmic forms.
		\end{proof}
		
		The following theorem shows that the Hodge structures that we have just defined are indeed the functorial Hodge structures defined by Deligne.
		
		\begin{thm}\label{functoriality}
		Let~$U$ be a smooth quasi-projective variety over~$\C$.
		\begin{enumerate}
		\item There exists a smooth projective variety~$X$ and an open immersion~$U\hookrightarrow X$ such that the complement~$L=X\setminus U$ is a hypersurface arrangement in~$X$.
		\item Given two such compactifications~$(X_1,L_1)$ and~$(X_2,L_2)$, the mixed Hodge structures on~$H^\bullet(U)$ defined via~$(X_1,L_1)$ and~$(X_2,L_2)$ are the same.
		\item The mixed Hodge structure on~$H^\bullet(U)$ defined in Theorem \ref{mhs} is the same as the mixed Hodge structure defined by Deligne in~\cite{delignehodge2}.
		\end{enumerate}
		\end{thm}
		
		\begin{proof}
		\begin{enumerate}
		\item This follows from Nagata's compactification theorem and Hironaka's resolution of singularities. In fact, we can assume that~$L$ is a normal crossing divisor.
		\item Using resolution of singularities, we can always embed~$U$ in a smooth projective variety~$X$ such that~$X\setminus U=L$ is a simple normal crossing divisor (and hence a hypersurface arrangement), and such that there exists morphisms
		$$(X_1,X_1\setminus L_1)\leftarrow (X,X\setminus L)\rightarrow (X_2,X_2\setminus L_2)$$
		that are the identity on~$U$. Hence by functoriality the two mixed Hodge structures are isomorphic to the mixed Hodge structure defined via~$(X,L)$.
		\item The claim follows from (2) and the fact that for a given~$U$, one can always choose~$(X,L)$ such that~$L$ is a normal crossing divisor (using resolution of singularities).
		\end{enumerate}
		\end{proof}
		
	\subsection{The Orlik-Solomon spectral sequence}\label{gysinss}
		
		Let~$X$ be a smooth projective variety and~$L$ be a hypersurface arrangement in~$X$. In the previous paragraph we defined a cohomological mixed Hodge complex on~$X$ that defines a mixed Hodge structure on the cohomology of~$X\setminus L$. The general formalism of mixed Hodge complexes (Theorem \ref{mhc}) tells us that the \emph{Orlik-Solomon spectral sequence}~$\wE_r^{p,q}$ associated to the weight filtration degenerates at~$E_2$. In this section we make the~$E_1$ term explicit. We will write~$\wE_r^{p,q}=\wE_r^{p,q}(X,L)$ when confusion might occur.\\
		
		By definition we have~$\wE_1^{-p,q}=\mathbb{H}^{-p+q}(\gr_p^W\K_\Q(X,L))$. From the proof of Theorem \ref{mhs} we get
		$$\wE_1^{-p,q}\cong \bigoplus_{S\in\s_p(L)} H^{-2p+q}(S)(-p)\otimes A_S(L).$$
		
		We first study the functoriality of the Orlik-Solomon spectral sequence.
		
		\begin{prop}\label{propfunctoriality}
		Let~$L$ (resp.~$L'$) be a hypersurface arrangement in a smooth projective variety~$X$ (resp.~$X'$), and~$\varphi:X\rightarrow X'$ a holomorphic map such that~$\varphi^{-1}(L')\subset L$. Let~$S$ and~$S'$ be strata of codimension~$p$ respectively of~$L$ and~$L'$ such that~$\varphi(S)\subset S'$ and let us denote by~$\varphi_{S,S'}:S\rightarrow S'$ the restriction of~$\varphi$. Then the component of the morphism~$$\wE_1^{-p,q}(\varphi):\wE_1^{-p,q}(X',L')\rightarrow \wE_1^{-p,q}(X,L)$$ indexed by strata~$S$ and~$S'$ is obtained by tensoring the morphism (\ref{functorialityS})
		$$A_{S,S'}(\varphi):A_{S'}(L')\rightarrow A_S(L)$$
		with the pull-back morphism
		$$\varphi_{S,S'}^*:H^{-2p+q}(S')\rightarrow H^{-2p+q}(S).$$
		
		The other components of~$\wE_1^{-p,q}(\varphi)$ are zero.
		\end{prop}
		
		\begin{proof}
		It is enough to do the proof over~$\C$ and work with the complexes~$\Omega^\bullet_{\langle X,L\rangle}$. There is a pull-back morphism
		$$\varphi^{-1}\Omega^\bullet_{\langle X',L'\rangle}\rightarrow \Omega^\bullet_{\langle X,L\rangle}$$
		that is compatible with the weight filtrations. Via the isomorphisms of Theorem \ref{grglobal}, one sees by local computation that this pull-back is as described in the Proposition at the level of holomorphic forms.
		\end{proof}
		
		When applied to the diagonal morphism~$X\rightarrow X\times X$, one gets an algebra structure on the~$E_1$ term of the Orlik-Solomon spectral sequence, as follows.
		
		\begin{prop}\label{product}
		The product		
		\begin{equation}\label{prodE}
		\wE_1^{-p,q}\otimes\wE_1^{-p',q'}\rightarrow\wE_1^{-(p+p'),q+q'}
		\end{equation}
		is obtained by tensoring the product morphisms (\ref{eqprodglobal})
		$$A_S(L)\otimes A_{S'}(L)\rightarrow A_{T}(L)$$
		with the morphisms
		$$H^{-2p+q}(S)\otimes H^{-2p'+q'}(S')\rightarrow 
		H^{-2p+q}(T)\otimes H^{-2p'+q'}(T)\overset{\cup}{\rightarrow}
		H^{-2(p+p')+(q+q')}(T)$$
		multiplied by the sign~$(-1)^{pq'}$. The above morphism is the composition of the restriction morphisms for the inclusion of~$T$ inside~$S$ and~$S'$, followed by the cup-product on~$T$.
		\end{prop}
		
		Note the sign~$(-1)^{pq'}$, which is a Koszul sign associated to the interchanging of the terms~$A_S(L)$ and~$H^{-2p'+q'}(S')$.
		
		We now turn to the description of the differential of the~$E_1$ term of the Orlik-Solomon spectral sequence.
		
		\begin{prop}\label{differential}
		Let~$S\subset S'$ be an inclusion of strata of~$L$ with~$\mathrm{codim}(S)=p$ and~$\mathrm{codim}(S')=p-1$. Then the component of the differential	
		$$d_1:\wE_1^{-p,q}\rightarrow\wE_1^{-p+1,q}$$
		indexed by~$S$ and~$S'$ is obtained by tensoring the natural morphism (\ref{eqderivglobal})
		$$A_S(L)\rightarrow A_{S'}(L)$$
		with the Gysin morphism
		$$H^{-2p+q}(S)(-p)\rightarrow H^{-2p+q+2}(S')(-p+1)$$ multiplied by the sign~$(-1)^{q-1}$. The other components of~$d_1$ are zero.
		\end{prop}
		
		\begin{proof}
		\textit{First step}: If~$L=D=\{D_1,\ldots,D_l\}$ is a normal crossing divisor, this is~\cite[Proposition 8.34]{voisin}, see also~\cite[Proposition 4.7]{peterssteenbrink}. Indeed in this case we have for every subset~$I\subset\{1,\ldots,l\}$,~$A_{D_I}(D)=\Q\, e_I$ a one-dimensional vector space.\\
		\textit{Second step}: We deduce the general case from the functoriality of the Orlik-Solomon spectral sequence and the fact that~$A_\bullet(L)$ is spanned by monomials~$e_I$ with~$I$ independent. Let~$e_I$ be such a monomial and let us write~$L(I)=\bigcup_{i\in I} L_i$, which is a normal crossing divisor in~$X$. From the functoriality of the spectral sequence, there is a map of spectral sequences
		$$\wE_1^{-p,q}(X,L(I))\rightarrow \wE_1^{-p,q}(X,L)$$
		which is easily seen to be injective (this follows from the injectivity in the deletion-restriction short exact sequence). Thus the differential of an element in~$H^{-2p+q}(S)\otimes \Q\, e_I$ can be read off~$\wE_1^{p,q}(X,L(I))$. We are then reduced to the first step.
		\end{proof}
		
		\begin{rem}
		 If~$X$ is any complex manifold, then we can also consider the Orlik-Solomon spectral sequence converging to the cohomology of~$X\setminus L$, and the above discussion for the~$E_1$ term remains valid. The only thing that we gain when assuming that~$X$ is a projective variety is the degeneracy of this spectral sequence at the~$E_2$ term, by Theorem \ref{mhc}.
		\end{rem}

	\subsection{The Orlik-Solomon model and the main theorem}\label{defM}
		
		We restate the results of the previous paragraph. Let~$X$ be a smooth projective variety and~$L$ a hypersurface arrangement in~$X$. Let us define
		$$M_q^n(X,L)=\bigoplus_{S\in\s_{q-n}(L)} H^{2n-q}(S)(n-q) \otimes A_S(L)$$
		viewed as a Hodge structure of weight~$q$.\\
		
		\begin{enumerate}
		\item We have a \textit{product} 
		\begin{equation}\label{prodM}
		M_q^n(X,L)\otimes M_{q'}^{n'}(X,L)\rightarrow M_{q+q'}^{n+n'}(X,L).
		\end{equation}
		obtained by tensoring the product morphisms (\ref{eqprodglobal})
		$$A_S(L)\otimes A_{S'}(L)\rightarrow A_{T}(L)$$
		with the morphisms
		$$H^{2n-q}(S)\otimes H^{2n'-q'}(S')\rightarrow 
		H^{2n-q}(T)\otimes H^{2n'-q'}(T)\overset{\cup}{\rightarrow}
		H^{2(n+n')-(q+q')}(T)$$
		multiplied by the sign~$(-1)^{(q-n)q'}$. The above morphism is the composition of the restriction morphisms for the inclusion of~$T$ inside~$S$ and~$S'$, followed by the cup-product on~$T$.	

		\item We have a \textit{differential}
		\begin{equation}\label{diffM}
		d:M_q^n(X,L)\rightarrow M_q^{n+1}(X,L).
		\end{equation}		
		Let~$S\subset S'$ be an inclusion of strata of~$L$ with~$\mathrm{codim}(S)=q-n$ and~$\mathrm{codim}(S')=q-(n+1)$. Then the component of the differential (\ref{diffM}) indexed by~$S$ and~$S'$ is obtained by tensoring the natural morphism (\ref{eqderivglobal})
		$$A_S(L)\rightarrow A_{S'}(L)$$
		with the Gysin morphism
		$$H^{2n-q}(S)(n-q)\rightarrow H^{2n-q+2}(S')(n-q+1)$$ multiplied by the sign~$(-1)^q$. The other components of the differential (\ref{diffM}) are zero.
		
		\item Let~$X'$ be another smooth projectiver variety,~$L'$ be a hypersurface arrangement in~$X'$ and~$\varphi:X\rightarrow X'$ be a holomorphic map such that~$\varphi^{-1}(L')\subset L$. Then we define a map 
		\begin{equation}\label{functorialityM}
		M^\bullet(\varphi):M^\bullet(X',L')\rightarrow M^\bullet(X,L).
		\end{equation}
		Let~$S$ and~$S'$ be strata of codimension~$q-n$ respectively of~$L$ and~$L'$ such that~$\varphi(S)\subset S'$, and let~$\varphi_{S,S'}:S\rightarrow S'$ be the restriction of~$\varphi$. Then the component of~$M^n_q(\varphi)$ indexed by~$S$ and~$S'$ is obtained by tensoring the morphism (\ref{functorialityS})
		$$A_{S,S'}(\varphi):A_{S'}(L')\rightarrow A_S(L)$$
		with the pull-back morphism
		$$\varphi_{S,S'}^*:H^{2n-q}(S')\rightarrow H^{2n-q}(S).$$
		The other components of~$M^\bullet(\varphi)$ are zero.
		\end{enumerate}		
		
		In the next theorem, a \emph{split mixed Hodge structure} is a mixed Hodge structure that is a direct sum of pure Hodge structures. Recall that a graded algebra~$B=\oplus_{n\geq 0}B_n$ is said to be graded-commutative if for homogeneous elements~$x$ and~$x'$ in~$B$ we have~$xx'=(-1)^{|x||x'|}x'x$.
		
		\begin{thm}\label{maintheorem}
		Let~$X$ be a smooth projective variety over~$\C$ and~$L$ be a hypersurface arrangement in~$X$.
		\begin{enumerate}
		\item The direct sum~$M^\bullet(X,L)=\bigoplus_{q}M^\bullet_q(X,L)$ is a graded-commutative differential graded algebra in the category of split mixed Hodge structures. It is functorial with respect to~$(X,L)$, using (\ref{functorialityM}).
		\item We have isomorphisms of algebras in the category of split mixed Hodge structures:~$$\gr^W H^{\bullet}(X\setminus L)\cong H^\bullet(M^\bullet(X,L)) .$$ They are functorial with respect to~$(X,L)$.
		\end{enumerate}
		\end{thm}
		
		We call~$M^\bullet(X,L)$ the \emph{Orlik-Solomon model} of the pair~$(X,L)$.
		
		\begin{proof}[Proof of Theorem \ref{maintheorem}]
		\begin{enumerate}
		 \item The assertion is a consequence of the previous paragraph (Propositions \ref{product}, \ref{differential} and \ref{propfunctoriality}). Note that we have multiplied the differential by~$-1$ for the sake of convenience; this gives an isomorphic differential graded algebra. 
		 \item The isomorphism is just, after the change of variables~$n=-p+q$, the fact that the spectral sequence~$\wE_r^{p,q}$ degenerates at~$E_2$ and converges to the cohomology of~$X\setminus L$:
		~$$H^{p}(\wE_1^{-\bullet,q})\cong\gr^W_q H^{-p+q}(X\setminus L).$$
		 \end{enumerate}
		\end{proof}
		
		\begin{rem}\label{remGysinquotient}
		Under the assumption (\ref{connectedassumption}), we may give a presentation of the Orlik-Solomon model that is more suitable in certain situations. For~$S$ a stratum of~$L$ and~$I\subset\{1,\ldots,l\}$ an independent subset such that~$L_I=S$, we have a monomial~$e_I\in A_S(L)$. If we identify~$H^{2n-q}(S)\otimes \Q\,e_I = H^{2n-q}(L_I)$, then we see that~$M_q^n(X,L)$ is the quotient of
		$$\bigoplus_{\substack{|I|=q-n\\I\textnormal{ indep.}}}H^{2n-q}(L_I)(n-q)$$
		by the sub-vector space spanned by the images of the morphisms
		$$H^{2n-q}(L_{I'})\rightarrow\bigoplus_{\substack{i\in I'\\ I'\setminus\{i\}\textnormal{ indep.}}} H^{2n-q}(L_{I'\setminus\{i\}})$$
		for~$I'$ dependent. The above morphism the alternate sum of identity morphisms (if~$I'$ is dependent and~$I'\setminus\{i\}$ is independent, then~$L_{I'\setminus\{i\}}=L_{I'}$ for dimension reasons).
		\end{rem}
		
\section{Wonderful compactifications and the Orlik-Solomon model}\label{sectionwonderful}

	\subsection{Hypersurface arrangements and wonderful compactifications}\label{secblowups}

		\begin{defi}\label{defigood}
		Let~$L=\{L_1,\ldots,L_l\}$ be a  hyperplane arrangement in~$\C^n$ and let~$Z$ be a stratum of~$L$. We say that~$Z$ is a \emph{good stratum} if there exists coordinates~$(z_1,\ldots,z_n)$ on~$\C^n$ such that~$Z=\{z_1=\cdots=z_r=0\}$ for some~$r$, and for each~$i=1,\ldots,l$,~$L_i$ is either of the type~$\{a_1z_1+\cdots+a_rz_r=0\}$ or of the type~$\{a_{r+1}z_{r+1}+\cdots+a_nz_n=0\}$.
		\end{defi}
		
		\begin{ex}
		In~$\C^3$, let~$L_1=\{x=0\}$,~$L_2=\{y=0\}$,~$L_3=\{z=0\}$,~$L_4=\{x=y\}$. Then the stratum~$\{x=y=0\}$ is good, but the stratum~$\{x=z=0\}$ is not.
		\end{ex}
		
		Let~$L=\{L_1,\ldots,L_l\}$ be a hypersurface arrangement in a complex manifold~$X$ and let~$Z$ be a stratum of~$L$. We say that~$Z$ is a \emph{good stratum} 
		if in every local chart where the~$L_i$'s are hyperplanes, it is a good stratum 
	 in the sense of the above definition. A stratum of dimension~$0$ (a point) is always good. In the case of a normal crossing divisor, all strata are good.
		
		\begin{lem}\label{blowuphyparr}
		Let~$L=\{L_1,\ldots,L_l\}$ be a hypersurface arrangement in a complex manifold~$X$,~$Z$ be a good stratum of~$L$, and
		$$\pi:\td{X}\rightarrow X$$ be the blow-up of~$X$ along~$Z$. Let~$E=\pi^{-1}(Z)$ be the exceptional divisor, and for all~$i$, let~$\td{L}_i$ be the strict transform of~$L_i$. Then~$\td{L}=\{E,\td{L}_1,\ldots,\td{L}_l\}$ is a hypersurface arrangement in~$\td{X}$.
		\end{lem}
		
		\begin{proof}
		It is enough to do the proof for~$X=\Delta^n$ and the~$L_i$'s hyperplanes. We choose coordinates~$(z_1,\ldots,z_n)$ as in Definition \ref{defigood}. We have~$r$ natural local charts~$\td{X}_k\cong\Delta^n$ on~$\td{X}$,~$k=1,\ldots,r$. On the chart~$\td{X}_k$, the blow-up morphism is given by
		$$\pi(z_1,\ldots,z_n)=(z_1z_k,\ldots,z_{k-1}z_k,z_k,z_{k+1}z_k,\ldots,z_rz_k,z_{r+1},\ldots,z_n)$$
		In this chart,~$E$ is defined by the equation~$z_k=0$. The strict transform of a hyperplane of the type~$\{a_1z_1+\cdots+a_rz_r=0\}$ is given by the equation~$a_1z_1+\cdots+a_{k-1}z_{k-1}+a_k+a_{k+1}z_{k+1}+\cdots+a_rz_r=0$. The strict transform of a hyperplane of the type~$\{a_{r+1}z_{r+1}+\cdots+a_nz_n=0\}$ is defined by the same equation.
		
		To sum up, in the chart~$\td{X}_k$, all the hypersurfaces of~$\td{L}$ are given by affine equations. Up to some translations, we can then find smaller charts where all the equations are linear. This completes the proof.
		\end{proof}
		
		With the notations of the above lemma, we will simply write that
		$$\pi:(\td{X},\td{L})\rightarrow (X,L)$$ is the blow-up of the pair~$(X,L)$ along the good stratum~$Z$. We stress the fact that~$\td{L}$ is the hypersurface arrangement consisting of the exceptional divisor~$E$ and all the proper transforms~$\td{L}_i$ of the hypersurfaces~$L_i$.\\
		
		The blow-ups along good strata 
		are enough to resolve the singularities of a hypersurface arrangement, as the following theorem shows. It is simply a reformulation of classical results on \enquote{wonderful compactifications}~\cite{fultonmcpherson,deconciniprocesi,hu,li}.
		
		\begin{thm}\label{seqblowups}
		 Let~$L$ be a hypersurface arrangement in a complex manifold~$X$. There exists a sequence
		~$$(\td{X},\td{L})= (X^{(N)},L^{(N)}) \overset{\pi_N}{\longrightarrow} (X^{(N-1)},L^{(N-1)}) \overset{\pi_{N-1}}{\longrightarrow} \cdots \overset{\pi_1}{\longrightarrow} (X^{(0)},L^{(0)})=(X,L)$$
		 where
		 \begin{enumerate}
		  \item for all~$k$,~$X^{(k)}$ is a complex manifold and~$L^{(k)}$ a hypersurface arrangement in~$X^{(k)}$
		  \item for all~$k$,~$\pi_k:(X^{(k)},L^{(k)})\rightarrow (X^{(k-1)},L^{(k-1)})$ is the blow-up of~$(X^{(k-1)},L^{(k-1)})$ along a 
		  good stratum of~$L^{(k-1)}$
		  \item~$\td{L}$ is a normal crossing divisor in~$\td{X}$.
		 \end{enumerate}
		\end{thm}
		
		\begin{proof}
		An arrangement of hypersurfaces defines an arrangement of subvarieties in the sense of~\cite{li}. Let us fix a building set~$\mathcal{G}$ and let~$\pi:\td{X}\rightarrow X$ be the corresponding wonderful compactification, with~$\td{L}=\pi^{-1}(L)$. Then according to~\cite{li},~$\pi$ is a composition of blow-ups along a minimal element of a building set. It simply remains to prove that a minimal element of a building set is a good stratum. We work in the cotangent spaces, hence reducing to a statement of linear algebra.
		
		Let~$\mathcal{G}$ be a building set of an arrangement of subspaces~$\mathcal{C}$ in the context of~\cite{deconciniprocesi}, and let us write~$M=\sum_{C\in\mathcal{C}}C$. We have a~$\mathcal{G}$-decomposition
		$$M=G_1\oplus\cdots\oplus G_r$$
		where the~$G_i\in\mathcal{G}$ are the maximal elements. Let~$X\in\mathcal{C}$ be any element, then~$X\subset M$ and by definition of a building set~$X\subset G_i$ for some unique~$i=1,\ldots, r$. Hence if we write~$U_i=\bigoplus_{j\neq i}G_j$, we then have, for all~$X\in\mathcal{C}$,~$X\not\subset G_i \Rightarrow X\subset U_i$.
		\end{proof}

	\subsection{Functoriality of the Orlik-Solomon model with respect to blow-ups}
	
		Let us consider a sequence of blow-ups along good strata as in Theorem \ref{seqblowups}:
		$$(\td{X},\td{L})= (X^{(N)},L^{(N)}) \overset{\pi_N}{\longrightarrow} (X^{(N-1)},L^{(N-1)}) \overset{\pi_{N-1}}{\longrightarrow} \cdots $$
		$$\cdots \overset{\pi_2}{\longrightarrow} (X^{(1)},L^{(1)})\overset{\pi_1}{\longrightarrow} (X^{(0)},L^{(0)})=(X,L).$$
		Then by the functoriality of the Orlik-Solomon model we get a sequence of morphisms of differential graded algebras (in the category of split mixed Hodge structures):
		$$M^\bullet(X,L)=M^\bullet(X^{(0)},L^{(0)}) \overset{\substack{M^\bullet(\pi_1)\\\sim}}{\longrightarrow} M^\bullet(X^{(1)},L^{(1)}) \overset{\substack{M^\bullet(\pi_2)\\\sim}}{\longrightarrow}\cdots$$ 
		$$\cdots \overset{\substack{M^\bullet(\pi_{N-1})\\\sim}}{\longrightarrow}  M^\bullet(X^{(N-1)},L^{(N-1)}) \overset{\substack{M^\bullet(\pi_{N})\\\sim}}{\longrightarrow}M ^\bullet(X^{(N)},L^{(N)})=M^\bullet(\td{X},\td{L}).$$
		For each~$k$,~$M^\bullet(\pi_k)$ is a quasi-isomorphism since~$\pi_k$ induces an isomorphism~$X^{(k)}\setminus L^{(k)}\isomto X^{(k-1)}\setminus L^{(k-1)}$. Thus we get a natural quasi-isomorphism between the Orlik-Solomon model of~$(X,L)$ and that of~$(\td{X},\td{L})$.
		
		In the following theorem, we give explicit formulas in the case of a single blow-up. For simplicity, we work under the assumption (\ref{connectedassumption}) and use the presentation of the Orlik-Solomon model given in Remark \ref{remGysinquotient}.
	
		\begin{thm}\label{formulaMpi}
		Let~$X$ be a smooth projective variety over~$\C$ and~$L$ be a hypersurface arrangement in~$X$ such that the assumption (\ref{connectedassumption}) is satisfied. Let~$Z$ be a good stratum of~$L$ and 
		$$\pi:(\td{X},\td{L})\rightarrow(X,L)$$
		be the blow-up of~$(X,L)$ along~$Z$. Let 
		$$M^\bullet(\pi):M^\bullet(X,L)\rightarrow M^\bullet(\td{X},\td{L})$$
		be the morphism of differential graded algebras induced by~$\pi$ on the Orlik-Solomon models. Then
		\begin{enumerate}
		\item~$M^\bullet(\pi)$ is a quasi-isomorphism.
		\item the components of~$M^n_q(\pi)$ are given, for~$I=\{i_1<\cdots<i_{q-n}\}$ independent, by
			\begin{enumerate}
				\item the pull-back morphism~$H^{2n-q}(L_I)\overset{\pi^*}{\rightarrow} H^{2n-q}(\td{L}_I)$.
				\item for all~$s$ such that~$Z\subset L_{i_s}$, the morphism~$H^{2n-q}(L_I)\rightarrow H^{2n-q}(E\cap\td{L}_{I\setminus\{i_s\}})$
				which is the pull-back morphism corresponding to~$E\cap\td{L}_{I\setminus\{i_s\}}\stackrel{\pi}{\rightarrow} Z\cap L_{I\setminus\{i_s\}}=Z\cap L_I\hookrightarrow L_I$, multiplied by the sign~$(-1)^{s-1}$.
			\end{enumerate}
		\end{enumerate}
		\end{thm}

		\begin{proof}
		\begin{enumerate}
		 \item This is obvious by Theorem \ref{maintheorem}, since~$\pi$ induces an isomorphism~$\td{X}\setminus\td{L}\isomto X\setminus L$.
		 \item It is a consequence of the general formula for functoriality given in \S\ref{defM}. Using the notation~$E=\td{L}_0$, a local computation shows that we have the following formula for~$A_\bullet(\pi):A_\bullet(L)\rightarrow A_\bullet(\td{L})$.
		~$$A_1(\pi)(e_i)=\begin{cases} e_i & \textnormal{if }L_i\textnormal{ does not contain } S\\ e_0+e_i & \textnormal{if }L_i\textnormal{ contains } S\end{cases}$$
		 \end{enumerate}
		 Thus we get
		~$$A_\bullet(\pi)(e_I)=e_I+\sum_{\substack{1\leq s\leq q-n\\Z\subset L_{i_s}}}(-1)^{s-1}e_0\wedge e_{I\setminus\{i_s\}}$$
		 and the claim follows.
		\end{proof}
		
		The above theorem and the work of Morgan~\cite[Theorem 10.1]{morganalgebraictopology} imply that~$M^\bullet(X,L)$ is a model of the space~$X\setminus L$ in the sense of rational homotopy theory.
		
		\begin{thm}\label{thmrationalhomotopy}
		The differential graded algebra $M^\bullet(X,L)$ and the space $X\setminus L$ have (non-canonically) isomorphic minimal models.
		\end{thm}
		
\section{Configuration spaces of points on curves}\label{confcurves}

	\subsection{Configuration spaces associated to graphs}
		 
		Let~$Y$ be a compact Riemann surface, i.e. a smooth projective complex curve. Let~$\Gamma$ be a finite unoriented graph with no multiple edges and no self-loops, with~$V$ its set of vertices and~$E$ its set of edges. Let~$Y^V$ be the cartesian power of~$Y$ indexed by~$V$, with coordinates~$y_v$. For~$v\in V$, we have a projection~$$p_v:Y^V\rightarrow Y.$$
		Every edge~$e\in E$ with endpoints~$v$ and~$v'$ defines a diagonal~$\Delta_e=\{y_v=y_{v'}\}\subset Y^V$ which is the locus where the coordinates corresponding to the two endpoints of~$e$ are equal. We define~$\Delta_\Gamma=\bigcup_{e\in E}\Delta_e$ and then the configuration space of points on~$Y$ associated to~$\Gamma$:
		$$C(Y,\Gamma)=Y^V\setminus\Delta_\Gamma.$$
		In the case where~$\Gamma=K_n$ is the complete graph on~$n$ vertices, we recover the configuration space 
		
		\begin{equation*}
		C(Y,n)=\{(y_1,\ldots,y_n)\in Y^n\,\,|\,\,y_i\neq y_j\ \,\, \textnormal{for}\,\, i\neq j\}=Y^n\setminus\bigcup_{i<j}\Delta_{i,j}.
		\end{equation*}
		
	\subsection{A model for the cohomology}
		
		In~\cite{kriz} and~\cite{totaroconf}, I. Kriz and B. Totaro independently found a model for the cohomology of~$C(Y,n)$. Their result has been recently generalized to~$C(Y,\Gamma)$ by S. Bloch in~\cite{blochtreeterated} (even though Bloch's framework is slightly more general, with external edges in~$\Gamma$ labeled by points of~$Y$). We recall the definition of this model. Here~$Y$ has dimension~$1$, but the general definition is similar.\\
		
		If~$B=\oplus_{n\geq 0}B_n$ is a graded-commutative graded algebra and~$\{x_\alpha\}$ are indeterminates with prescribed degrees~$\{d_\alpha\}$, then there is a well-defined notion of graded-commutative algebra generated by the~$x_\alpha$'s over~$B$. This is a graded-commutative graded algebra which is the quotient of~$B[\{x_\alpha\}]$ by the relations~$bx_\alpha=(-1)^{|b|d_\alpha}x_\alpha b$ for~$b$ homogeneous, and~$x_\beta x_\alpha=(-1)^{d_\alpha d_\beta}x_\alpha x_\beta$ for all~$\alpha$ and~$\beta$. For example, if~$B$ is a field concentrated in degree~$0$ then we recover the exterior algebra generated by the~$x_\alpha$'s. We use the wedge notation~$x_\alpha\wedge x_\beta$ to remember the graded-commutativity property.\\
		
		Let us define, following~\cite{blochtreeterated}, a graded-commutative differential graded algebra~$N^\bullet(Y,\Gamma)$ in the following way. It is generated (as a graded-commutative algebra) by the cohomology~$H^\bullet(Y^V)$ and elements~$G_e$ in degree~$1$ for every edge~$e\in E$, modulo the relations:
		\begin{enumerate}[(R1)]
		 \item~$p_{v}^*(c)G_e=p_{v'}^*(c)G_e$ for every class~$c\in H^\bullet(Y)$, where~$v$ and~$v'$ are the endpoints of~$e$ in~$\Gamma$.
		 \item~$\sum_{i=1}^r(-1)^{i-1}G_{e_1}\wedge\cdots\wedge\widehat{G_{e_i}}\wedge\cdots\wedge G_{e_r}=0$ if~$\{e_1,\ldots,e_r\}\subset E$ contains a loop.
		\end{enumerate}
		\vspace{5mm}

		We now define a differential~$d$ on~$N^\bullet(Y,\Gamma)$ as zero on~$H^\bullet(Y^V)$ and given on the elements~$G_e$ by the formula~$$d(G_e)=[\Delta_e]\in H^2(Y^V).$$
		One shows that~$d$ is well-defined and makes~$N^\bullet(Y,\Gamma)$ into a graded-commutative differential graded algebra.

	\subsection{The isomorphism with the Orlik-Solomon model}
		
		By choosing charts on~$Y$, one easily sees that~$L=\Delta_\Gamma$ is a hypersurface arrangement in~$X=Y^V$. Thus theorem \ref{maintheorem} can be applied to the pair~$(Y^V,\Delta_\Gamma)$ and gives a model for the cohomology of~$C(Y,\Gamma)=Y^V\setminus\Delta_\Gamma$. We fix an linear order on the set~$E$ of edges of~$\Gamma$, hence on the irreducible components~$\Delta_e$ of~$\Delta_\Gamma$. This allows us to consider the Orlik-Solomon model~$M^\bullet(Y^V,\Delta_\Gamma)$, with its presentation given by Remark \ref{remGysinquotient}. Thus~$M_q^n(Y^V,\Delta_\Gamma)$ is a quotient of 
		$$\bigoplus_{\substack{I\subset E\\ |I|=q-n\\I\textnormal{ indep.}}} H^{2n-q}(\Delta_I)(n-q).$$
		We note that a subset~$I\subset E$ is dependent if and only if it contains a loop, and is a circuit if and only if it is a simple loop.\\
		
		We define a morphism of differential graded algebras~$$\alpha:N^\bullet(Y,\Gamma)\rightarrow M^\bullet(Y^V,\Delta_\Gamma)$$ in the following way.\\
		First we note that for all~$n$ we have~$M^n_n(Y^V,\Delta_\Gamma)=H^n(Y^V)$, and we easily see that the resulting (injective) map~$H^\bullet(Y^V)\rightarrow M^\bullet(Y^V,\Delta_\Gamma)$ is a map of graded algebras. Then we define~$\alpha(G_e)$ to be a generator~$g_e$ of~$H^0(\Delta_e)(-1)\subset M^1_2(Y^V,\Delta_\Gamma)$. 
		
		\begin{lem}
		 The morphism~$\alpha$ is well-defined and compatible with the differentials. It is thus a map of differential graded algebras.
		\end{lem}
		
		\begin{proof}
		 First we show that~$\alpha$ respects relations~$(\mathrm{R1})$ and~$(\mathrm{R2})$. For relation~$(\mathrm{R1})$ we see that by definition
		~$$\alpha(p_v^*(c)G_e)=p_v^*(c)g_e=p_v^*(c)_{|\Delta_e}\in H^\bullet(\Delta_e).$$
		 This equals~$i_e^*(p_v^*(c))=(p_v\circ i_e)^*(c)$ where~$i_e:\Delta_e\hookrightarrow Y^V$ is the inclusion of~$\Delta_e$. The relation then follows from the equality~$p_v\circ i_e=p_{v'}\circ i_e$.\\
		 For relation~$(\mathrm{R2})$ we can assume that we have~$e_1<\cdots<e_r$. Then if~$R$ is the expression in the relation~$(\mathrm{R2})$ we have
		~$$\alpha(R)=\sum_{i=1}^r(-1)^{i-1}g_{e_1}\cdots\widehat{g_{e_i}}\cdots g_{e_r}$$
		 and~$g_{e_1}\cdots\widehat{g_{e_i}}\cdots g_{e_r}$ is a generator of~$H^0(\Delta_{e_1}\cap\cdots\cap\widehat{\Delta_{e_i}}\cap\cdots\cap\Delta_{e_r})(-r+1)$. Since~$\{\Delta_{e_1},\ldots,\Delta_{e_r}\}$ is dependent,~$\alpha(R)$ is thus killed by the quotient that defines~$M^\bullet(Y^V,\Delta_\Gamma)$.\\
		 We then show that~$\alpha$ is compatible with the differentials. By definition, the differential is zero on~$H^\bullet(Y^V)\subset M^\bullet(Y^V,\Delta_\Gamma)$. Furthermore,~$d\alpha(G_e)=d(g_e)$ is, by definition of the Gysin morphism, the class of~$\Delta_e$ in~$H^2(Y^V)$. This completes the proof.
		\end{proof}

		\begin{thm}\label{compbloch}
		 The morphism~$\alpha:N^\bullet(Y,\Gamma)\rightarrow M^\bullet(Y^V,\Delta_\Gamma)$ is an isomorphism of differential graded algebras.
		\end{thm}	
			
		\begin{proof}
		 We sketch the proof and leave the details to the reader. We define the inverse morphism~$\beta$ in the following way.
		 Let~$I\subset E$ be an independent set of edges of~$\Gamma$ of cardinality~$|I|=q-n$, let~$i_I:\Delta_I\hookrightarrow Y^V~$ be the inclusion of the corresponding stratum. Let~$f_I:Y^V\rightarrow \Delta_I$ be any natural splitting of~$i_I$ defined out of projections~$p_v$'s. Then we define the component of~$\beta$:
		~$$\beta^n_q:H^{2n-q}(\Delta_I)\rightarrow H^{2n-q}(Y^V)G_I$$ to be the pull-back~$f_I^*$. The degrees match since~$H^{2n-q}(Y^V)G_I$ is in degree~$2n-q+|I|=n$. It remains to prove that~$\beta$ passes to the quotient that defines~$M^\bullet(Y^V,\Delta_\Gamma)$, and defines an inverse to~$\alpha$.
		\end{proof}
		
		\begin{rem}
		 It is striking that Kriz and Totaro's model works for configuration spaces of points on any smooth projective variety~$Y$, where the diagonals can have any codimension. It is then tempting to ask for a generalization of the Orlik-Solomon model to the cohomology of~$X\setminus L$ where~$L\subset X$ locally looks like a union of sub-vector spaces of any codimension inside~$\C^n$. In~\cite{totaroconf}, B. Totaro suggests a particular case of the previous question, focusing on vector spaces~$V_i$ of a fixed codimension~$c$ such that all intersections~$V_{i_1}\cap\cdots\cap V_{i_r}$ have codimension a multiple of~$c$ (the present article handles the case~$c=1$).
		\end{rem}
		
	\subsection{Comparison with Kriz's quasi-isomorphism}\label{seccompkriz}
	
		In this paragraph we sketch the proof that Kriz's quasi-isomorphism~$\varphi$ from~\cite{kriz} can be recovered as a consequence of the functoriality of the Orlik-Solomon model.\\
		For the sake of convenience we use the notations from~\cite{kriz} and write~$E^\bullet(n)$ for~$N^\bullet(Y,K_n)$ where~$K_n$ is the complete graph on~$n$ vertices. We write~$\Delta=\Delta_{K_n}$ for the union of all diagonals of~$Y^n$. According to Theorem \ref{compbloch}, we have an isomorphism of differential graded algebras
		$$\alpha:E^\bullet(n)\stackrel{\cong}{\rightarrow}M^\bullet(Y^n,\Delta).$$
		
		Let~$\pi:Y[n]\rightarrow Y^n$ be the Fulton-MacPherson wonderful compactification~\cite{fultonmcpherson}. Then~$D=\pi^{-1}(\Delta)$ is a simple normal crossing divisor whose irreducible components~$D(S)$ are indexed by subsets~$S\subset\{1,\ldots,n\}$ with~$|S|\geq 2$. We now describe the model~$F^\bullet(n)$ defined by Kriz. By its very definition~\cite[\S 6]{fultonmcpherson}, we have a natural isomorphism of differential graded algebras
		$$\varepsilon:F^\bullet(n)\stackrel{\cong}{\longrightarrow}M^\bullet(Y[n],D)$$
		between~$F^\bullet(n)$ and the Orlik-Solomon model~$M^\bullet(Y[n],D)$. To make this isomorphism precise, let us mention that
		\begin{itemize}
		\item on~$H^\bullet(Y^n)$,~$\varepsilon$ is the pull-back~$H^\bullet(\pi):H^\bullet(Y^n)\rightarrow H^\bullet(Y[n])$;
		\item~$\varepsilon(S)$ is the generator~$g_S\in H^0(D(S))(-1)$ and~$\varepsilon(D_S)$ is the class~$[D(S)]\in H^2(Y[n])$. 
		\end{itemize}
	
		\begin{thm}
		We have a commutative square
		\begin{displaymath}\xymatrix{
		F^\bullet(n) \ar[r]^-{\varepsilon} & M^\bullet(Y[n],D) \\
		E^\bullet(n) \ar[r]_-{\alpha} \ar[u]^{\varphi} & M^\bullet(Y^n,\Delta) \ar[u]_{M^\bullet(\pi)}\\
		}\end{displaymath}
		where~$\varphi$ is defined in~\cite[\S 3]{kriz}, the horizontal arrows are isomorphisms of differential graded algebras and the vertical arrows are quasi-isomorphisms of differential graded algebras.
		\end{thm}	
		
		\begin{proof}
		It only remains to prove that we have
		$$M^1(\pi)(g_{a,b})=\sum_{S\supset\{a,b\}}g_S.$$
		We do the proof in the case~$n=3$ (the cases~$n<3$ being trivial) and leave the general case to the reader. We may assume that~$\{a,b\}=\{1,2\}$. Then~$\pi$ is simply the blow-up along~$\Delta_{1,2,3}$,~$D(1,2,3)$ is the exceptional divisor, and the equality~$M^1(\pi)(g_{1,2})=g_{1,2}+g_{1,2,3}$ is a consequence of Theorem \ref{formulaMpi}.
		\end{proof}

\bibliographystyle{alpha}
\bibliography{bibliodethese}

\begin{thebibliography}{CHKS06}

\bibitem[Alu12]{aluffi}
P.~Aluffi.
\newblock Chern classes of free hypersurface arrangements.
\newblock {\em J. Singul.}, 5:19--32, 2012.

\bibitem[Arn69]{arnold}
V.~I. Arnol'd.
\newblock The cohomology ring of the group of dyed braids.
\newblock {\em Mat. Zametki}, 5:227--231, 1969.

\bibitem[BH14]{bibbyhilburnchordal}
C.~Bibby and J.~Hilburn.
\newblock Quadratic-linear duality and rational homotopy theory of chordal
  arrangements.
\newblock {\em preprint: ar{X}iv:1409.6748}, 2014.

\bibitem[Bib13]{bibbyabelian}
C.~Bibby.
\newblock Cohomology of abelian arrangements.
\newblock {\em preprint: ar{X}iv:1310.4866}, 2013.

\bibitem[Blo12]{blochtreeterated}
S.~Bloch.
\newblock Motives, the fundamental group, and graphs.
\newblock {\em preprint}, 2012.

\bibitem[Bri73]{brieskorn}
E.~Brieskorn.
\newblock Sur les groupes de tresses [d'apr{\`e}s {V}. {I}. {A}rnol'd].
\newblock In {\em S{\'e}minaire {B}ourbaki, 24{\`e}me ann{\'e}e (1971/1972),
  {E}xp. {N}o. 401}, pages 21--44. Lecture Notes in Math., Vol. 317. Springer,
  Berlin, 1973.

\bibitem[CHKS06]{catanesehostenkhetansturmfels}
F.~Catanese, S.~Hosten, A.~Khetan, and B.~Sturmfels.
\newblock {The maximum likelihood degree.}
\newblock {\em {Am. J. Math.}}, 128(3):671--697, 2006.

\bibitem[DCP95]{deconciniprocesi}
C.~De~Concini and C.~Procesi.
\newblock Wonderful models of subspace arrangements.
\newblock {\em Selecta Math. (N.S.)}, 1(3):459--494, 1995.

\bibitem[Del71]{delignehodge2}
P.~Deligne.
\newblock Th{\'e}orie de {H}odge. {II}.
\newblock {\em Inst. Hautes {\'E}tudes Sci. Publ. Math.}, (40):5--57, 1971.

\bibitem[Del74]{delignehodge3}
P.~Deligne.
\newblock Th{\'e}orie de {H}odge. {III}.
\newblock {\em Inst. Hautes {\'E}tudes Sci. Publ. Math.}, (44):5--77, 1974.

\bibitem[{Dol}07]{dolgachevlogarithmic}
I.~V. {Dolgachev}.
\newblock {Logarithmic sheaves attached to arrangements of hyperplanes.}
\newblock {\em {J. Math. Kyoto Univ.}}, 47(1):35--64, 2007.

\bibitem[FM94]{fultonmcpherson}
W.~Fulton and R.~MacPherson.
\newblock A compactification of configuration spaces.
\newblock {\em Ann. of Math. (2)}, 139(1):183--225, 1994.

\bibitem[Get99]{getzlerresolvingmhm}
E.~Getzler.
\newblock Resolving mixed {H}odge modules on configuration spaces.
\newblock {\em Duke Math. J.}, 96(1):175--203, 1999.

\bibitem[Hu03]{hu}
Y.~Hu.
\newblock A compactification of open varieties.
\newblock {\em Trans. Amer. Math. Soc.}, 355(12):4737--4753, 2003.

\bibitem[Kri94]{kriz}
I.~Kriz.
\newblock On the rational homotopy type of configuration spaces.
\newblock {\em Ann. of Math. (2)}, 139(2):227--237, 1994.

\bibitem[Ler59]{leray}
J.~Leray.
\newblock Le calcul diff{\'e}rentiel et int{\'e}gral sur une vari{\'e}t{\'e}
  analytique complexe. ({P}robl{\`e}me de {C}auchy. {III}).
\newblock {\em Bull. Soc. Math. France}, 87:81--180, 1959.

\bibitem[Li09]{li}
L.~Li.
\newblock Wonderful compactification of an arrangement of subvarieties.
\newblock {\em Michigan Math. J.}, 58(2):535--563, 2009.

\bibitem[Loo93]{looijenga}
E.~Looijenga.
\newblock Cohomology of {${\mathscr{M}}_3$} and {${\mathscr{M}}^1_3$}.
\newblock In {\em Mapping class groups and moduli spaces of {R}iemann surfaces
  ({G}{\"o}ttingen, 1991/{S}eattle, {WA}, 1991)}, volume 150 of {\em Contemp.
  Math.}, pages 205--228. Amer. Math. Soc., Providence, RI, 1993.

\bibitem[Mor78]{morganalgebraictopology}
J.~W. Morgan.
\newblock The algebraic topology of smooth algebraic varieties.
\newblock {\em Inst. Hautes {\'E}tudes Sci. Publ. Math.}, (48):137--204, 1978.

\bibitem[OS80]{orliksolomon}
P.~Orlik and L.~Solomon.
\newblock Combinatorics and topology of complements of hyperplanes.
\newblock {\em Invent. Math.}, 56(2):167--189, 1980.

\bibitem[OT92]{orlikterao}
P.~Orlik and H.~Terao.
\newblock {\em Arrangements of hyperplanes}, volume 300 of {\em Grundlehren der
  Mathematischen Wissenschaften [Fundamental Principles of Mathematical
  Sciences]}.
\newblock Springer-Verlag, Berlin, 1992.

\bibitem[PS08]{peterssteenbrink}
C.~A.~M. Peters and J.~H.~M. Steenbrink.
\newblock {\em Mixed {H}odge structures}, volume~52 of {\em Ergebnisse der
  Mathematik und ihrer Grenzgebiete. 3. Folge. A Series of Modern Surveys in
  Mathematics [Results in Mathematics and Related Areas. 3rd Series. A Series
  of Modern Surveys in Mathematics]}.
\newblock Springer-Verlag, Berlin, 2008.

\bibitem[Sai80]{saitologarithmic}
K.~Saito.
\newblock Theory of logarithmic differential forms and logarithmic vector
  fields.
\newblock {\em J. Fac. Sci. Univ. Tokyo Sect. IA Math.}, 27(2):265--291, 1980.

\bibitem[Ter78]{teraologarithmic}
H.~Terao.
\newblock Forms with logarithmic pole and the filtration by the order of the
  pole.
\newblock In {\em Proceedings of the {I}nternational {S}ymposium on {A}lgebraic
  {G}eometry ({K}yoto {U}niv., {K}yoto, 1977)}, pages 673--685, Tokyo, 1978.
  Kinokuniya Book Store.

\bibitem[Tot96]{totaroconf}
B.~Totaro.
\newblock Configuration spaces of algebraic varieties.
\newblock {\em Topology}, 35(4):1057--1067, 1996.

\bibitem[Voi02]{voisin}
C.~Voisin.
\newblock {\em Hodge theory and complex algebraic geometry. {I}}, volume~76 of
  {\em Cambridge Studies in Advanced Mathematics}.
\newblock Cambridge University Press, Cambridge, 2002.
\newblock Translated from the French original by Leila Schneps.

\bibitem[Yuz01]{yuzvinskiorliksolomon}
S.~Yuzvinski{\u\i}.
\newblock Orlik-{S}olomon algebras in algebra and topology.
\newblock {\em Uspekhi Mat. Nauk}, 56(2(338)):87--166, 2001.

\end{thebibliography}
\end{document}